\documentclass [12pt]{amsart}
\usepackage[utf8]{inputenc}
\usepackage[all]{xy}
\usepackage[T2A,T1]{fontenc}
\usepackage[russian,english]{babel}
\usepackage{amssymb, enumerate}
\usepackage{mathrsfs}
\usepackage[lite, initials]{amsrefs}
\usepackage{geometry}
\usepackage{hyperref}
\usepackage{xcolor}
\usepackage{dcpic,pictexwd}

\newtheorem{teorema}{Theorem}[section]
\newtheorem*{theorem*}{Main Theorem}
\newtheorem{lemma}[teorema]{Lemma}
\newtheorem{propos}[teorema]{Proposition}
\newtheorem{corol}[teorema]{Corollary}
\theoremstyle{definition}
\newtheorem{ex}{Example}[section]
\newtheorem{rem}{Remark}[section]
\newtheorem{defin}[teorema]{Definition}
\newtheorem{assumption}[teorema]{Assumption}

\calclayout
\makeatletter 
\def\@tocline#1#2#3#4#5#6#7{\relax
  \ifnum #1>\c@tocdepth 
  \else
    \par \addpenalty\@secpenalty\addvspace{#2}%
    \begingroup \hyphenpenalty\@M
    \@ifempty{#4}{%
      \@tempdima\csname r@tocindent\number#1\endcsname\relax
    }{%
      \@tempdima#4\relax
    }%
    \parindent\z@ \leftskip#3\relax \advance\leftskip\@tempdima\relax
    \rightskip\@pnumwidth plus4em \parfillskip-\@pnumwidth
    #5\leavevmode\hskip-\@tempdima
      \ifcase #1
       \or\or \hskip 1em \or \hskip 2em \else \hskip 3em \fi%
      #6\nobreak\relax
    \dotfill\hbox to\@pnumwidth{\@tocpagenum{#7}}\par
    \nobreak
    \endgroup
  \fi}
\makeatother

\newcommand{\Zhe}{\mbox{\usefont{T2A}{\rmdefault}{m}{n}\CYRZH}}

\def\R{{\mathbb R}}

\def\C{{\mathbb C}}

\def\Z{{\mathbb Z}}

\def\CP{{\mathbb{CP}}}
\def\H{{\mathbb H}}
\def\cU{{\mathcal U}}
\def\cJ{{\mathcal J}}

\def\sfera{{\mathbb S}}

\DeclareMathOperator{\pv}{\wedge \mspace{-9.5mu}_\star \ }

\def\Re{{\sf Re}}
\renewcommand{\Im}{\mathsf{Im}}

\def\Aut{\operatorname{Aut}}

\newcommand{\vecpart}[1]{\underline{#1}}
\newcommand{\vecnorm}[1]{\underline{#1}^2}

\title{Slice regular functions as covering maps and global $\star$-roots}
\author[A. Altavilla]{Amedeo Altavilla}\address{Altavilla Amedeo:  Dipartimento di Matematica, Universit\`a degli Studi di Bari ``Aldo Moro'', via Edoardo Orabona, 4, 70125,
Bari, Italy.}\email{amedeo.altavilla@uniba.it}
\author[S. Mongodi]{Samuele Mongodi}\address{Mongodi Samuele: Università degli Studi di Milano-Bicocca, Dipartimento di Matematica e Applicazioni,via Roberto Cozzi, 55, 20125, Milano, Italy.}\email{samuele.mongodi@unimib.it}

\thanks{Partially supported by GNSAGA of INdAM and by the INdAM project ``Teoria delle funzioni ipercomplesse e applicazioni''.}

\subjclass[2010]{Primary 30G35, 30C25; secondary 58K10, 32A10}
\keywords{Slice-regular functions, $k$-th roots, covering maps, monodromy}

\begin{document}
\begin{abstract}
The aim of this paper is to prove that a large class of quaternionic slice regular functions result to be (ramified) covering maps. By means of the topological 
implications of this fact and by providing further 
topological structures, we are able to give suitable natural conditions for the existence of $k$-th
$\star$-roots of a slice regular function. Moreover, we are also able to compute all the solutions which, quite surprisingly, in the most general case, are in number of $k^2$. The last part is devoted to compute the monodromy and to present a technique to compute all the $k^2$ roots
starting from one of them. 
\end{abstract}

\maketitle
\tableofcontents

\section{Introduction}

The present work aims at studying slice regular functions of a quaternionic variable as covering maps and the existence and nature of global $k$-th $\star$-roots of slice functions.

\vspace{.5cm}

Geometric function theory is the study of geometric and topological properties of analytic function of a complex variable; one could very well say that one of the fundamental (although rather simple) results that originate and motivate such a study is the local nature of branched coverings of holomorphic functions of one variable.

This characteristic is even more striking when we move from open planar domains to Riemann surfaces: holomorphic maps between Riemann surfaces are locally branched coverings. In other terms, we can always find local coordinates such that a holomorphic map between Riemann surfaces is locally written as $z\mapsto z^k$, with $k\in\mathbb{N}$.

The global counterpart of these statements is obtained, for instance, in the case when the function has some finiteness properties: as an example, polynomial functions are branched coverings of $\C$ over $\C$, whose branching set is related to the zeros of their derivative. The monodromy around branching points can be quite complicated and it is related to the Galois group of the corresponding extension of rings of rational functions, as opposed to the rather simple behaviour of the local model $z\mapsto z^k$. However, if we are interested in the problem of lifting holomorphic functions via such coverings, even the local model presents interesting phenomena: if $p$ is a holomorphic polynomial, the Riemann surface $w^k=p(z)$ possesses a number of geometric features that are relevant in the general study of Riemann surfaces.

\vspace{.5cm}

This geometric description of the local behaviour of holomorphic maps in one dimension brings together several ingredients: Rouché theorem, winding numbers, logarithmic indicator, local invertibility, conformality, nature of zeroes.

In the setting of slice regular functions of a quaternionic variable, all these results are, to some degree, true, however the difficulties which are inherent in the quaternionic theory prevent us from merging them in a comprehensive geometric description of slice regular functions as mappings from $\H$ to $\H$.

We are referring to the absence of the usual (quaternionic) product and composition operations, which do not leave the set of slice regular functions invariant. 

The case of the product is quite representative of the challenges posed by the quaternionic setting: given two slice regular functions $f,g:\H\to\H$ their product $q\mapsto f(q)g(q)$ (where, on the right, we consider the product operation in the algebra of the quaternions) is not a slice regular function; this problem is overcome by considering a suitable product, called $\star$-product (see Definition \ref{def_starprod}). However, the value of $f\star g$ at a quaternion $q$ is not, apart from some particular cases, the product of the values of $f$ and $g$ at $q$, nor can be obtained from these two values alone.

This results in a cumbersome way of dealing with powers and exponentials.

\vspace{.5cm}

Quite recently, a number of papers addressed the problem of finding the analogues of a logarithm or a $k$-th root for slice regular functions, see \cites{AltavillaPAMS,altavillaLOG,GPV,GPV2}.

Our purpose, in the present work, is to study the existence and nature of th $k$-th $\star$-roots of a slice regular function, i.e. the solutions of $g^{\star k}=f$, from the point of view of covering maps, thus obtaining global results and allowing a study of the monodromy of the solution of the functional equation $g^{\star k}=f$.

Clearly, many results on $k$-th $\star$-roots could be derived out of those obtained for the $\star-\log$. For instance, using natural ideas, it should be possible to prove many results on the existence of a $\star$-root of a function, starting from the possible existence of its $\star$-logarithm.
However, the main difference with other previous attempts at this task lays in the techniques we employ, which stem from merging our two different, but related, interpretations of slice regularity, developed in our respective previous works in the field.

On the one hand, any slice regular function can be interpreted, via a suitable complex analytic representation, as a holomorphic curve in $\C^4\equiv \C\otimes\H$, so we can apply all the techniques of classical complex analysis and relate geometric properties of such
a curve to characteristic of the slice regular function (see ~\cites{Mongodi:HolSR,Mongodi:AssAlg}). In this setting, the $\star$-product emerges naturally as the operation induced on such curves by the algebra product of $\C\otimes\H$.

On the other hand, the space of slice regular functions can be given the structure of a rank $4$ module over the set of multiplicative commutators, i.e. the set of slice preserving functions $\mathcal{S}_\R$ (see~\cites{AltavillaAMPA,AltavillaLAA}); this construction depends on the choice of a basis $(1,\ I,\ J,\ K)$ of $\H$ as a real vector space, so that, given $f_0,\ f_1,\ f_2,\ f_3$ slice preserving, the map
$$(f_0,\ f_1,\ f_2,\ f_3)\mapsto f_0+f_1I+f_2J+f_3K$$
is a bijection. In this case, the $\star$-product is recovered by observing that the slice preserving functions are such that $f\star g=g\star f=fg=gf$ and this gives us a way to extend the product in $\H$ to $(\mathcal{S}_\R)^4$.

\vspace{.5cm}

As different in spirit as they seem, these two viewpoints are, in fact, two sides of the same coin; while the complex analytic approach is useful in giving clear and general proofs based on known techniques in complex analysis and geometry, the algebraic approach closely relates the peculiar characteristics of quaternions to the properties of slice regular functions, particularly when the operations of the algebra structure are involved, making it easier to understand the computational side and allowing to produce several explicit examples. 

They are therefore both useful in separating those behaviours which come seamlessly out of the theory of one complex variable from the phenomena that are properly caused by the unique properties of $\H$, thus revealing the true r\^ole of the quaternions.

We believe that the combination of these two approaches could be useful in dealing with other similar problems and, particularly, in exploring the geometric implications of such results, venturing in the scarcely explored realm of quaternionic Riemann surfaces.

\vspace{.7cm}

The content is organized as follows.

\vspace{.5cm}

Section \ref{secpreliminaries} contains a review of the basic material needed for our purposes, following the notations and strategies introduced inn~\cites{GP:AltAlg,Mongodi:HolSR}. In particular, we introduce the formalism of stem functions and of slice preserving functions and their relations with the $\star$-product.

In the last part of the section, namely in Subsection~\ref{twoapproaches}, 
we recall the general ideas leading the authors to
develop the two aforementioned interpretations
of slice regularity, emphasizing how these are linked to each other and showing part of their potential.

\vspace{.5cm}

In Section~\ref{coversect} we show that, under
suitable natural hypotheses, a slice regular function that is also a finite map is, in fact, a covering map (see Theorem~\ref{teoremacover}). 

To obtain this result we present a couple of technical lemmas characterizing geometrically, in terms of tangent vectors, the non-invertibility of
the real differential of a slice regular function (see Lemma~\ref{lmm_tang} and~\ref{lmm_zeroD}). Of course, part of the content of these results was already known in the literature (maybe with different notations), but the proofs we present
here are new (especially in giving geometric insights on the involved objects), and the techniques will be
exploited in the mentioned main result. 

As a consequence, we obtain a new proof (based on covering properties of slice regular functions), of the fact that an injective slice regular function has real differential that is invertible in the whole domain (see~\cites{GSS:Twistor,Altavilla:RealDiff,Ghiloni2020}). To offer the right perspective
on Theorem~\ref{teoremacover}, we underline that 
from the content of~\cites{Ghiloni2020},
it can be already inferred that slice regular functions are \textit{locally} covering maps, while
our new techniques give a \textit{global} result.

\vspace{.5cm}

The rest of the paper is devoted to study $k$-th
$\star$-roots. 

We begin, in Section~\ref{secpower},
with the more abstract case arising from the complex analytic interpretation of slice
regularity. In particular, having transferred the notion of $\star$-product to curves in $\C^4$, 
we coherently define the analog of the $k$-th $\star$-power, denote by $\sigma_k$. In
Lemma~\ref{lmm_jacobianpower} we give an explicit
expression of the differential of $\sigma_k$,
with a complete description of its zero set. 

This
is related to a result given in the last section of~\cites{AltavillaAMPA}, where the authors give
algebraic conditions in order to have that the $k$-th $\star$-power of a slice regular function
is slice preserving.

Passing to a more abstract interpretation, we are then able to prove that, under suitable hypotheses, $\sigma_k$ is generically $k^2$-to-$1$ and it 
is in fact a covering map (see Proposition~\ref{prp_preimages} and Theorem~\ref{teo_powercover}). In the last
part of this section we show how the hypotheses
of the previously mentioned results are read in the language of slice regular functions.

\vspace{.5cm}

Section~\ref{starrootsec} contains the main outcomes of this paper
as mentioned in the abstract. On the basis of the results of the previous section, we are able to prove that, under suitable natural hypotheses, any slice regular function defined on a domain without real points admits $k^2$ $k$-th $\star$-roots (see Theorem~\ref{casosenzareali}).

This result was quite unexpected but is not so surprising once one realizes that the space of
quaternions minus the real line $\H\setminus\R$ is
biholomorphic to a half complex plane times the Riemann sphere $\C^+\times\mathbb{CP}^1$. In a certain sense, these two basic complex spaces produces two independent monodromies, each one counting $k$ sheets. If we impose the condition
for the domain to intersect the real axis, then many of these solutions are not anymore well defined and the amount of survivors is the more
expected number $k$ (see Theorem~\ref{teo_kroots}).

In the last part of this section we provide two explicit examples showing, in a heuristic way, the different behaviours between the case $k$ odd and $k$ even. In both examples, a key role is played by
the so-called \textit{slice polynomial functions} introduced in~\cites{AltavillaMATHZ}.

\vspace{.5cm}

The last section is devoted to compute the monodromy in detail, explaining, in particular, the $k^2$ factor discussed before. We find that,
for a fixed $k$, it is possible to define two
different actions of the set of $k$-th rooths of unity in $\C$ on the set of $k$-th $\star$-roots
of a given slice regular function. In terms of the algebraic representation of slice regular functions, these two actions are related to two different representations of complex numbers (as a field extension of $\R$ and as a subset of $2\times 2$ real matrices).

We explore in
detail these two actions, highlighting their different nature in the two cases, $k$ odd and $k$ even; this allows us to obtain a general existence result for $k$-th roots of slice functions (see Theorems \ref{teo_genrootsodd} and \ref{teo_genrootseven}).

We notice that this \textit{double} monodromy, producing the $k^2$ factor, is coherent with the results
obtained for the logarithm, where, under consistent natural hypotheses, a given slice regular function has $\infty^2$ $\star$-logarithms (see~\cite[Theorem 1.2]{altavillaLOG}).

\medskip
We warmly thank the anonymous referees for their useful comments which helped to improve the presentation of
our results. 
\section{Preliminaries on quaternionic slice regular functions}\label{secpreliminaries}
We refer to~\cites{GSS:RegFunc} for a general introduction to the slice regularity.

Let $\H$ be the algebra of quaternions and $(1,i,j,k)$ its standard basis satisfying usual multiplicative rules. Then, any quaternion $q\in\H$ can be written as 
$q=q_{0}+q_{1}i+q_{2}j+q_{3}k$, with $q_{0},q_{1},q_{2},q_{3}\in\R$. We endow $\H$ with
the standard involution $\H\ni q=q_{0}+q_{1}i+q_{2}j+q_{3}k\mapsto q^{c}=q_{0}-(q_{1}i+q_{2}j+q_{3}k)\in\H$. The euclidean norm of $q$ can therefore be computed as $||q||=\sqrt{qq^{c}}$.
Moreover, thanks to such a conjugation, it is possible to define the scalar and vector part
of any $q\in\H$ as
$$
q_{0}=\frac{q+q^{c}}{2}\;,\qquad q_{v}=\frac{q-q^{c}}{2}\;,
$$
respectively. 
Therefore, any $q\in\H$ can be also written as $q=q_{0}+q_{v}$. 
Then, we will identify $\R\subset\H$ with the set $\{q\in\H\ :\ q_v=0\}$ and $\R^3$ with the set of purely imaginary quaternions $\{q\in\H\ :\ q_0=0\}$.
With this 
representation, the product of two quaternions $q=q_{0}+q_{v}$ and $p=p_{0}+p_{v}$ can be written as
$$
qp=q_{0}p_{0}-\langle q_{v},p_{v}\rangle+q_{0}p_{v}+p_{0}q_{v}+q_{v}\wedge p_{v}\;,
$$
where $\langle\cdot,\cdot\cdot\rangle$ and $\wedge$ denote the standard euclidean and
cross product in $\R^{3}$.

The set of imaginary units in $\H$ is diffeomorphic to a $2$-sphere $S^{2}$ and
will be denoted as follows
$$\sfera=\{I\in\H\ :\ I^2=-1\}=\{q\in\H\ :\ q_{0}=0,\ ||q_{v}||=1\}\;.$$
With this notation, any $q=q_{0}+q_{v}\in\H\setminus\R$ can be written as
$q=\alpha+I\beta$, where $\alpha=q_{0}$, $\beta=||q_{v}||$ and $I=q_{v}/||q_{v}||\in\sfera$
and hence $\H=\cup_{I\in\sfera}\C_{I}$, where $\C_{I}=Span_\R(1,I)$.
Clearly, for any $\alpha,\beta\in\R$ and any $I\in\sfera$, $\alpha+I\beta$ is a well defined quaternion. Therefore, we can define the map $\pi:\C\times\sfera\to\H$ by
\begin{equation}\label{pi}
\pi(z,I)=\Re(z)+I\Im(z)\;.
\end{equation}
We will denote the imaginary unit in $\C$ by $\imath$.
Notice that, for any $\alpha+\imath\beta\in\C$, $\pi(\{z\}\times\sfera)=\{\alpha+I\beta\ :\ I\in\sfera\}\simeq S^2$.

In view of the introduction of slice regularity, we now set up some material on the real tensor product $\C\otimes\H$. The imaginary unit
of this complexification will be denoted by $\sqrt{-1}$.
We will write the elements in $\C\otimes\H$ both as $z\otimes q$, with $z\in\C$ and $q\in\H$, or as $v_0+\sqrt{-1}v_1$, with $v_0,v_1\in\H$.
We extend the previously defined quaternionic conjugation in the following way: if $v_0+\sqrt{-1}v_1\in\C\otimes\H$, then $(v_0+\sqrt{-1}v_1)^{c}:=v_0^{c}+\sqrt{-1}v_1^{c}$.
Given $v=v_{0}+\sqrt{-1}v_{1}, w=w_{0}+\sqrt{-1}w_{1}\in\C\otimes\H$ their product is defined as
$$
vw=v_{0}w_{0}-v_{1}w_{1}+\sqrt{-1}(v_{0}w_{1}+v_{1}w_{0})\;.
$$

The map defined in Formula~\eqref{pi} induces a new map (denoted with the same symbol) $\pi:\C\otimes \H \times\sfera\to \H$ by requiring the linearity in the first component and that
\begin{equation}\label{pi2}
    \pi(z\otimes q, I)=\pi(z,I)q\;.
\end{equation}
In other words, if $w=w_0+\sqrt{-1}w_1\in\mathbb{C}\otimes\mathbb{H}$, with $w_1,w_2\in\mathbb{H}$, and $I\in\sfera$, then $$\pi(w,I)=w_0+Iw_1\;.$$
We recall that the complex structure defined by left multiplication by $\sqrt{-1}$ gives a structure of complex affine space to $\C\otimes\H$ for which it results
to be biholomorphic to $\C^4$. Having chosen the standard basis $(1,i,j,k)$, if $p=p_0+p_1i+p_2j+p_3k, q=q_0+q_1i+q_2j+q_3k\in\H$, then such a biholomorphism $\phi:\C\otimes\H\to\C^4$
easily reads as follows 
\begin{equation}\label{phimap}
    \phi(p+\sqrt{-1}q)\mapsto (p_0+\imath q_0,p_1+\imath q_1,p_2+\imath q_2,p_3+\imath q_3)\;.
\end{equation}

The complex conjugation extends to a linear map, still denoted in the same way, such that $\overline{z\otimes q}=\overline{z}\otimes q$,
or, if $w=w_0+\sqrt{-1}w_1$, then $\overline{w}=w_0-\sqrt{-1}w_1$. For any $w\in\C\otimes\H$, we have that
$(\overline{w})^{c}=\overline{(w^{c})}$.

We have now all the tools needed to introduce slice regularity with the approach of Ghiloni and Perotti \cite{GP:AltAlg}.

\begin{defin}
Let $\cU\subset\C$ be an open domain such that $\overline{\cU}=\cU$ and $U=\pi(\cU\times\sfera)$. 
A function $F:\cU\to \C\otimes\H$ such that $F(\overline{z})=\overline{F(z)}$ is said
to be a \emph{stem function}.
A function $f:U\to\H$ is called \emph{slice} if there exists a stem function $F:\cU\to\C\otimes \H$ such that $f(\pi(z,I))=\pi(F(z),I)$. We will write $f=\mathcal{I}(F)$.
Moreover, if $F$ is a holomorphic function, then $f$ is said to be \emph{slice regular}.
\end{defin}

For the convenience of what follows we pose here the following assumption setting up
some notation.

\begin{assumption}\label{assumption1}
In what follows $\cU\subset\C$ will be an open domain and $U=\pi(\cU\times\sfera)$. Moreover, we will fix a slice function $f=\mathcal{I}(F):U\to\H$ and we will write $F=1\otimes F_0+\sqrt{-1}\otimes F_1$ where $F_0,\ F_1:\cU\to \H$. In other words if $\alpha+\imath \beta\in\cU$, $F=F_0+\sqrt{-1}F_1$ and $f=\mathcal{I}(F)$, then
$f(\alpha+I\beta)=F_0(\alpha+\imath\beta)+IF_1(\alpha+\imath\beta)$.
\end{assumption}

As the pointwise product of two slice functions does not preserve sliceness, a different, yet natural, product may be introduced.

\begin{defin}\label{def_starprod}
Let $f=\mathcal{I}(F), g=\mathcal{I}(G)$ be two slice functions defined on $U$. We define
their \emph{slice product} or $\star$-\emph{product} as 
$$
f\star g:=\mathcal{I}(FG):U\to\H\;.
$$
\end{defin}
Hence if $F=F_{0}+\sqrt{-1}F_{1}$ and $G=G_{0}+\sqrt{-1}G_{1}$, then $f\star g$
is the slice function induced by $F_{0}G_{0}-F_{1}G_{1}+\sqrt{-1}(F_{0}G_{1}+F_{1}G_{0})$.

A special subset of slice regular functions is given in the following definition.

\begin{defin}
A slice function $f=\mathcal{I}(F):U\to\H$ is said to be \emph{slice preserving} if, 
for any $I\in\sfera$, $f(U\cap \C_I)\subset\C_{I}$. Equivalently, this happens if and only if 
$F$ takes values in $\C\otimes\R\subset\C\otimes\H$.
\end{defin}
A slice preserving function $f=\mathcal{I}(F_{0}+\sqrt{-1}F_{1})$ is therefore a slice function
such that $F_{0}$ and $F_{1}$ are real valued. However, this does not
imply that $f$ is real valued.
Notice that if $f$ is a slice preserving function and $g$ is any slice function, then
$f\star g=g\star f=fg$.

We conclude this section by defining the slice conjugate and symmetrized of a slice function.

\begin{defin}
Let $f=\mathcal{I}(F_{0}+\sqrt{-1}F_{1}):U\to\H$ be a slice function. We define its \emph{slice conjugate} as the slice function $f^{c}:=\mathcal{I}(F^{c}):U\to \H$ and
its \emph{symmetrized function} or \emph{symmetrization} as $f^{\mathsf{s}}=f\star f^{c}=\mathcal{I}(FF^{c})$.
\end{defin}
The symmetrized function of any slice function is a slice preserving function. Moreover, if $f$ is slice preserving, then $f^{c}=f$ and so $f^{\mathsf{s}}=f^{\star 2}=f^{2}$.

From the general theory~\cite{GSS:RegFunc,GP:AltAlg}, we know that if $f(q)=0$, then
$f^{\mathsf{s}}(q)=0$. In particular, if $q=\alpha+I\beta\in\H\setminus\R$, then $f^{\mathsf{s}}(\alpha+J\beta)=0$, for any $J\in\sfera$. Moreover, if $f^{\mathsf{s}}(\alpha+I\beta)=0$, then there exists $J\in\sfera$, such that $f(\alpha+J\beta)=0$.

\subsection{Two slightly different approaches}\label{twoapproaches}

The introduction of slice regularity by means of $\C\otimes\H$ and all the structures presented
above allows us to exploit complex analysis and geometry in order to obtain our results.

First of all the identification between $\C\otimes\H$ and $\C^4$ was already noticed in~\cite[Remark 3(2)]{GP:AltAlg} and used in~\cite[Theorem 3.4]{AltavillaCVEE} to prove
a result about a particular family of slice regular functions. Later, this approach was widely exploited in~\cites{Mongodi:HolSR,Mongodi:AssAlg} to produce alternative 
proofs of known and new results on slice regularity, highlighting their holomorphic nature. 

On the other hand, such an approach is an underlying element of a series of paper by
the first author and de Fabritiis~\cite{AltavillaAMPA,AltavillaPAMS,AltavillaLAA,altavillaLOG}.
In these papers many algebraic properties of slice regular functions are proved as well
as existence and uniqueness results for the quaternionic exponential and logarithm are
given. The basic result and idea behind these researches is a result due to Colombo, Gonzalez-Cervantes and Sabadini~\cite[Proposition 3.12]{CGCS} and Ghiloni, Moretti and Perotti~\cite[Lemma 6.11]{GMP} stating that any slice regular function $f:U\to\H$ can be 
uniquely written as a sum
$$
f=f_{0}+f_{1}i+f_{2}j+f_{3}k,
$$
where $f_{0},f_{1},f_{2},f_{3}$ are slice preserving regular functions defined on the same domain. Thanks to this equality, it is possible to define the ``real'' and ``vector'' part of $f=f_{0}+f_{v}$ as
$$
f_{0}=\frac{f+f^{c}}{2}\;,\qquad f_{v}=\frac{f-f^{c}}{2}\;,
$$
and, with these observation, it is possible to represent the $\star$-product of two slice regular functions
$f=f_{0}+f_{v}, g=g_{0}+g_{v}$ as
$$
f\star g=f_{0}g_{0}-\langle f_{v},g_{v}\rangle_{\star}+f_{0}g_{v}+g_{0}f_{v}+f_{v}\pv g_{v}\;,
$$
where $\langle\cdot,\cdot\cdot\rangle_{\star}$ and $\pv$ are algebraic operators acting, formally, as the usual euclidean and cross product (see~\cite{AltavillaAMPA}).
In particular, given $f=f_{0}+f_{1}i+f_{2}j+f_{3}k$, we have that $f^{c}=f_{0}-(f_{1}i+f_{2}j+f_{3}k)=f_0-f_v$, and $f^{\mathsf{s}}=f_{0}^{2}+f_{1}^{2}+f_{2}^{2}+f_{3}^{2}$.

It is clear that this interpretation of slice regular functions is suggested and underlined by
the previously described complex one. However, if the complex approach
will be effective to produce general proofs, the second one will be exploited to produce
and describe explicit examples.

\section{Differential  and covering properties of slice regular functions}\label{coversect}

In this section we recover some known fact about the real differential of slice regular function by looking it at the level of $\C\otimes \H$. This
approach is not only interesting by itself but will be exploited later in this section when we will discuss about covering properties.

With a slightly different language and notation, in \cite{Mongodi:HolSR} the second author  defined, for $q\in\H$, the sets
\begin{align*}
    Z_q&=\{w\in \C\otimes \H\ :\ \pi(w,I)=q\textrm{ for some }I\in\sfera\}\;,\\
\mathfrak{Z}_q&=\{(w,I)\in\C\otimes\H\times\sfera\ :\ \pi(w,I)=q\}\;,
\end{align*}
where $\pi$ is the map defined in Formula~\eqref{pi2}.

\begin{rem}\label{rem1}
We note that, if $w=w_0+\sqrt{-1}w_1\in\C\otimes\H$ is a real point in the complexification of $\H$, i.e. if $w_1=0$, then
$w\in Z_q$ if and only if $w=q$.
\end{rem}

As in \cite{Mongodi:AssAlg}, we also consider the map $\mathscr{F}:\cU\times\sfera\to\C\otimes \H\times\sfera$ given by $\mathscr{F}(z,I)=(F(z),I)$, obtaining the following equality $f\circ\pi=\pi\circ\mathscr{F}$ and commutative diagram.

$$
\begindc{\commdiag}[3]
\obj(0,180)[A]{$\cU\times\sfera$}
\obj(280,180)[B]{$\C\otimes \H\times\sfera$}
\obj(0,0)[C]{$U$}
\obj(280,0)[D]{$\H$}
\mor{A}{B}{$\mathscr{F}$}[\atleft,\solidarrow]
\mor{A}{C}{$\pi$}[\atright,\solidarrow]
\mor{C}{D}{$f$}[\atright,\solidarrow]
\mor{B}{D}{$\pi$}[\atleft,\solidarrow]
\enddc$$

Moreover, for any $(w,I)\in \C\otimes\H \times\sfera$ we have that $\ker D\pi_{(w,I)}=T_{\mathscr{F}(z,I)}\mathfrak{Z}_{f(q)}$.
It is known~\cite{Mongodi:HolSR} that $Z_q$ is a complex quadratic cone in $\C\otimes \H$, with one singular point (the vertex), which is the only point with real coordinates, and that $\mathfrak{Z}_q$ is a complex submanifold of $\C\otimes\H\times\sfera$, once $\sfera$ is given the appropriate complex structure (which is actually the natural one on $\CP^1$). As $q$ varies in $\H$, the manifolds $\mathfrak{Z}_q$ give a foliation of $\C\otimes\H\times\sfera$.

\begin{rem}\label{rem_Clinear}
Notice that the definition of slice regularity can be interpreted in terms of pseudoholomorphic curves as follows. Let $f:U\to\mathbb{H}$ be a slice regular function and $p=\alpha+I\beta\in U\setminus\mathbb{R}$. Locally, the tangent space $T_pU$ can be split into a sum $T_pU=T_{\alpha+\imath\beta}\mathbb{C}\oplus T_I\mathbb{S}$. Hence, if $v\in T_{\alpha+\imath\beta}\mathbb{C}$, then
$df(\imath v)=L_I(df(v))$, where $L_I$ denotes the left multiplication by $I\in\sfera$.
\end{rem}

In the following lemma we give a geometric interpretation for the differential of a
slice regular function to be singular.

\begin{lemma}\label{lmm_tang}Suppose that $q=\pi(z,I)\in U$ is such that $\Im(z)\neq 0$,  $F_1(z)\neq 0$ and that the (real) differential $Df_q$ is singular. Then $F(\cU)$ is tangent to $Z_{f(q)}$ in $F(z)$.\end{lemma}
\begin{proof}
If we set $\cU^+=\{z\in\cU\ :\ \Im(z)>0\}$, then the restriction of $\pi$ to $\cU^+\times\sfera$ is a diffeomorphism between the latter and $U\setminus\R$. Hence, $Df_q$ is singular if and only if $D(\pi\circ\mathscr{F})_{(z,I)}$ is singular, i.e. $Df_q[v]=0$ for some $v\in T_{q}U$ if and only if there exists $v'\in T_{(z,I)}(\cU\times\sfera)$ such that
$$D\pi_{\mathscr{F}(z,I)}\circ D\mathscr{F}_{(z,I)}[v']=0\;.$$
This happens if and only if $D\mathscr{F}_{(z,I)}[v']\in\ker D\pi_{\mathscr{F}(z,I)}$, i.e. by definition if and only if $D\mathscr{F}_{(z,I)}[v']\in T_{\mathscr{F}(z,I)}\mathfrak{Z}_{f(q)}$.

We now show that the former analysis implies that $F(\cU)$ is tangent to $Z_{f(q)}$.
Consider the projection map $p_1:\C\otimes\H\times\sfera\to\C\otimes\H$, so that $p_1\circ\mathscr{F}(z,I)=F(z)$, and let $Z_{f(q)}^*\subset\C\otimes\H$ denote the set $Z_{f(q)}$ minus its real points, i.e. $Z_{f(q)}^*=Z_{f(q)}\setminus\{f(q)\}$ (see Remark~\ref{rem1}). The map $p_1$ induces a biholomorphism between  $Z^*_{f(q)}$ and $\mathfrak{Z}^*_{f(q)}=p_1^{-1}(Z^*_{f(q)})$ (see the proof of \cite{Mongodi:HolSR}*{Theorem 3.3}). Therefore, as $F_1(z)\neq 0$,
$$0\neq Dp_1(D\mathscr{F}_{(z,I)}[v'])\in T_{p_1(\mathscr{F}(z,I))}Z_{f(q)}=T_{F(z)}Z_{f(q)}$$
so, as $p_1\circ\mathscr{F}(z,I)=F(z)$, $D(p_1\circ\mathscr{F})_{(z,I)}[v']=DF_z[v'']$, where $v''$ is the component of $v'$ along $T_z\C$; therefore $F(\cU)$ is tangent to $Z_{f(q)}$ in $F(z)$.
\end{proof}

\begin{rem}We consider the case when $F(\cU)\subseteq Z_{f(q)}$ as a tangency case.\end{rem}

From this point of view, we obtain another proof of a known result (see \cite{GSS:Twistor}*{Proposition 3.3} and~\cite{Ghiloni2020}) regarding
the invertibility of the real differential of a slice regular function. Before state it, we set up the following couple of notations.
For $I\in\sfera$, we identify $T_I\sfera\subset T_I\H\cong\H$ with
$$A_I=\{p\in \H\ :\ Ip+pI=0\}$$
(see \cite{Mongodi:AssAlg}*{Lemma 3.1}). Moreover, given $w\in\C\otimes \H$, $w=1\otimes a+\sqrt{-1}\otimes b$, we identify $T_w(\C\otimes\H)$ with $T_a\H\oplus T_b\H$.

\begin{lemma}\label{lmm_zeroD}Suppose that $q=\pi(z,I)$ with $\Im(z)\neq 0$. The differential $Df_q$ is singular if and only if one of the following three cases occurs:
\begin{enumerate}
\item $F_1(z)=0$;
\item $DF_z=0$;
\item $\pi(DF_z,I)F_1(z)^{-1}\in A_I$.
\end{enumerate}
\end{lemma}

\begin{rem}
Before performing the proof, notice that the three cases of the previous Lemma, correspond to $(1)$ the \textit{spherical derivative} of $f$ being zero
(see~\cite[Definition 1.18]{GSS:RegFunc}); $(2)$ the slice derivative being zero (see again~\cite[Definition 1.7]{GSS:RegFunc}); $(3)$ the product of the slice derivative with the inverse
of the spherical derivative belonging to the orthogonal complement of $Span(1,q)$ (see~\cite[Proposition 29]{Altavilla:RealDiff}). Indeed, for this last claim, notice that
$Span(1,q)^\bot=\{p\in\H\ :\ qp+pq=0\}$, and if $q=\alpha+I\beta$, the equality $qp+pq=0$ is equivalent to $Ip+pI=0$, which defines $A_I$.
\end{rem}

\begin{proof}

Consider $D\mathscr{F}_{(z,I)}:T_z\C\oplus T_I\sfera\to T_a\H\oplus T_b\H\oplus T_I\sfera$, where $a=F_0(z)$, $b=F_1(z)$; we have
$$D\mathscr{F}=\begin{pmatrix}DF_0&0\\DF_1&0\\0& id_I\end{pmatrix}$$
where $DF_0:T_z\C\to T_a\H$, $DF_1:T_z\C\to T_b\H$ are the real differentials of $F_0$ and $F_1$ and $id_I$ is the identity map on $T_I\sfera$.

Likewise, the differential $D\pi_{(w,I)}:T_a\H\oplus T_b\H\oplus T_I\sfera\to T_{\pi(w,I)}\H$ is given by
$$D\pi_{(w,I)}=\begin{pmatrix}id_a&L_I&R'_b\end{pmatrix}$$
where $id_a$ is the identity map on $T_a\H$, $L_I:\H\to\H$ is the left multiplication by $I$, i.e. $L_I(q)=Iq$, and $R'_b:A_s\to \H$ is the right multiplication by $b$, restricted to the subspace $A_I$, i.e. $R'_b(h)=hb$.

Given $v'\in T_z\C\oplus T_I\sfera$, we write it $v'=v'_1+v'_2$, with $v'_1\in T_z\C$ and $v'_2\in T_I\sfera$, then
$$D(\pi\circ\mathscr{F})_{z,I}[v']=(DF_0+L_IDF_1)_z[v'_1]+R_{F_1(z)}'[v'_2]\;.$$
Therefore, the rank of $Df$ drops in the following three cases:
\begin{enumerate}
\item there exists $v'_2\neq 0$ such that $v'_2\in\ker R'_{F_1(z)}$,
\item there exists $v'_1\neq 0$ such that $v'_1\in\ker (DF_0+L_IDF_1)_z$,
\item there exist $v'_1,v'_2\neq 0$ such that $(DF_0+L_I F_1)_z[v'_1]=R_{F_1(z)}'[v'_2]$.
\end{enumerate}
In the first case, as the map $R_b$ is always invertible, unless $b=0$, we conclude that $F_1(z)=0$. 
By Remark \ref{rem_Clinear}, we have that $(DF_0+L_IDF_1)[\imath v'_1]=L_I(DF_0+L_IDF_1)[v'_1]=0$, so $DF_z=0$.

In the third case, we are saying that there exists $v'_1$ such that
$$R_{F_1(z)}^{-1}(DF_0+L_IDF_1)_z[v'_1]\in A_I\;.$$
We note that, as the right multiplication operator commutes with the left one, we have that
$$R_{F_1(z)}^{-1}(DF_0+L_IDF_1)_z[\imath v'_1]=R_{F_1(z)}^{-1}L_I(DF_0+L_IDF_1)_z[v'_1]=L_IR_{F_1(z)}^{-1}(DF_0+L_IDF_1)_z[v'_1]$$
and the subspace $A_I$ is stable under $L_I$, therefore, as both the spaces have real dimension $2$, the third case happens if and only if $R_{F_1(z)}^{-1}(DF_0+L_IDF_1)$ is an isomorphism between $T_z\C$ and $A_I=T_I\sfera$.

By the definition of $A_I$, this happens if and only if
$$L_IR_{F_1(z)}^{-1}(DF_0+L_IDF_1)+R_IR_{F_1(z)}^{-1}(DF_0+L_IDF_1)=0\;,$$
i.e. if and only if $\pi(DF_z,I)F_1(z)^{-1}\in A_I$.
\end{proof}

We remark that cases (2) and (3) are both contained in the geometric statement of Lemma \ref{lmm_tang}, i.e. in both cases $F(z)$ is a point of tangency between $F(\cU)$ and $Z_{f(q)}$ (and a regular point for the latter).

Let us now consider the  holomorphic function $\Phi_q:\C\otimes\H\to\C$, where $q=q_0+q_1i+q_2j+q_3k$ is a quaternion, given by
$$\Phi_q(z_0+z_1i+z_2j+z_3k)=\sum_{h=0}^3(z_h-q_h)^2\;,$$
where $z_0+z_1i+z_2j+z_3k=\phi^{-1}(z_0,z_1,z_2,z_3)$.
The function $\Phi_q$ is such that its zero set equals $Z_q$ when embedded in $\C^4$ (see~\cite[Corollary 3.4]{Mongodi:HolSR}). 
Notice that, if $F:\cU\to\C\otimes\H$ is a map inducing a slice function $f=f_{0}+f_{1}i+f_{2}j+f_{3}k$, and $z=\alpha+\imath\beta\in\cU$, then 
$\Phi_{q}(F(z))=0$ if and only if $(F(z)-q)(F^{c}(z)-q^{c})=0$, if and only if 
$(f-q)^{\mathsf{s}}=(f_{0}-q_{0})^{2}+(f_{1}-q_{1})^{2}+(f_{2}-q_{2})^{2}+(f_{3}-q_{3})^{2}=0$.

As we will see in the
next corollary, properties of $\Phi_q$ relate with the rank of slice regular functions.

\begin{corol}\label{cor_mult}Given $q=\pi(z,I)\in\H$, the real differential $Df_q$ is singular if and only if $z$ is a zero of multiplicity greater than $1$ for the holomorphic function $\Phi_{f(q)}\circ F:\cU\to\C$.\end{corol}
\begin{proof}Since $Z_{f(q)}=\Phi_{f(q)}^{-1}(0)$ (see \cite[Corollary 3.4]{Mongodi:HolSR}), as long as $F_1(z)\neq 0$ the tangent space to $Z_{f(q)}$ at $F(z)$ is given by the equation $D(\Phi_{f(q)})_{F(z)}[v']=0$; therefore,
$$0=(\Phi_{f(q)}\circ F)'(z)[v']=D(\Phi_{f(q)})_{F(z)}DF_z[v']$$
if and only if $DF_z[v'_1]$ is contained in $T_{F(z)}Z_{f(q)}$ for all $v'_1\in T_z\C$, i.e if and only if $F(\cU)$ is tangent to $Z_{f(q)}$ at $F(z)$.

If $F_1(z)=0$, then $F(z)$ has all real components, hence it is the vertex of the cone $Z_{f(q)}$, which is the only singular point of $\Phi_{f(q)}$, which is a zero of multiplicity $2$, so the conclusion is trivial in this case.
\end{proof}

\begin{rem}
This last corollary is a complex analyitic interpretation of a result obtained
independently in the quaternionic setting stating, essentially, that a slice regular function $f$ is singular at a point $q_0$ if and only if, 
$f(q)=f(q_0)+(q-q_0)\star(q-\tilde{q_0})g(q)$ for some $\tilde{q_0}\in\sfera_{q_0}$ 
and some slice regular function $g$ (see \cite[Proposition 3.6]{GSS:Twistor} and \cite[Theorem 30]{Altavilla:RealDiff}.
\end{rem}


We conclude this section by showing that a large class of slice regular functions are \textit{global} covering maps outside their (real) singular locus.
The fact that slice regular functions are \textit{local} coverings can be excerpt already from the content of~\cites{GSS:Twistor,Altavilla:RealDiff,Ghiloni2020}.
Before stating the theorem, let us define the following sets related to singular points and values of a slice regular function.
We denote by $C_0(f)$ the set of critical points of $f$, i.e. 
$$C_0(f)=\{p\in U\,:\,Df_q\text{ is singular}\}\; ;$$
we consider its image under $f$, i.e. the set of critical values 
$$C(f)=f(C_0(f))\;,$$ 
and its saturation under $f$, i.e. the singular locus of $f$, 
$$S(f)=f^{-1}(C(f))\;.$$

\begin{teorema}\label{teoremacover}
Suppose $f$ is a slice regular finite map. Then, the restriction
$$f\vert_{U\setminus S(f)}:U\setminus S(f)\to f(U\setminus S(f))\subseteq \H\setminus C(f)$$
is a covering map.
\end{teorema}
\begin{proof}Take a point $p\in f(U\setminus S(f))$ and a relatively compact neigborhood $V$ of $p$ in $f(U\setminus S(f))$ such that $V\Subset\H\setminus C(f)$. We consider the family of holomorphic maps
$$\{\Phi_{p'}\circ F:\cU\to\C\}_{p'\in V}\;$$
If $p'\to p$, then $\Phi_{p'}\circ F$ converges uniformly on compact sets to $\Phi_p\circ F$.

Suppose that $f^{-1}(p)=\{q_1,\ldots, q_n\}$ and write $q_m=\pi(z_m,I_m)$. As $q_1,\ldots, q_n\in U\setminus S(f)$, the differential of $f$ is non singular around such points, so $f$ is locally injective and, by Corollary \ref{cor_mult}, $\Phi_p\circ F$ has a zero of multiplicity $1$ in each of these points.

Therefore, by Hurwitz's theorem, $\Phi_{p'}\circ F$ has $n$ zeroes (counted with multiplicities) for $p'\in V$ (up to shrinking $V$, if necessary). However, again by Corollary \ref{cor_mult}, as long as $p'\in \H\setminus C(f)$, the zeroes of $\Phi_{p'}\circ F$ have multiplicity $1$, so each point $p'\in V$ has exactly $n$ preimages.

This shows that $f^{-1}(V)$ is the disjoint union of $n$ open sets $V_1,\ldots, V_n$. The differential of $f$ is non singular on every $V_m$ and $f\vert_{V_m}$ is injective, so $f$ is a diffeomorphism between $V_m$ and $V$, for $m=1,\ldots, n$. This implies that $f$ is a covering map from $U\setminus S(f)$ to $f(U\setminus S(f))\subseteq\H\setminus C(f)$.
\end{proof}

Along the same lines of the proof, we obtain the following result, which already appeared, proved with different methods, in \cites{GSS:RegFunc, Altavilla:RealDiff,Ghiloni2020}.
The proof we are going to present differs from those in the cited papers for being more topological, rather than analytical, shedding new light on the geometry hidden in the background of the techniques we used in the previous results. Even if we will not be needing this corollary in the rest of the paper, we think some readers might find the proof helpful in understanding and contextualizing the machinery we introduced in this section.

\begin{corol}If $f:U\to\H$ is slice regular and injective, its real differential is invertible at all points.\end{corol}
\begin{proof}Suppose that the differential is not invertible at $q=\pi(z_0,I)$, then either $F_1(z_0)=0$ and then $f$ is constant on the sphere $\pi(\{z_0\}\times \sfera)$, or $\Phi_{f(q)}\circ F$ has a zero of multiplicity greater than $1$ in $z_0$.

If $\Phi_{f(q)}\circ F$ vanishes identically, then the function $f$ is constant on some $2$-dimensional surface in $U$ (see~\cite[Remark 3.4]{Mongodi:HolSR} and~\cite[Proposition 5.2]{Ghiloni2020}). If $\Phi_{f(q)}\circ F$ does not vanish identically, by Hurwitz's theorem, $\Phi_{q'}\circ F$ has more than one zero (counting multiplicities) for all $q'$ close enough to $f(q)$. However, suppose that there exists a neighborhood $V$ of $f(q)$ such that $\Phi_{q'}\circ F$ has a zero of multiplicity greater than one: this means that $\mathscr{F}(\cU\times\sfera)$ is tangent to $\mathfrak{Z}_{q'}$ for all $q'\in V$.

Without loss of generality, we can suppose that the tangency points are such that $F_1(z)\neq 0$. On $\C\otimes \H\times\sfera$, we consider the smooth subbundle $E$ of the tangent bundle given by
$$\C\otimes\H\times\sfera\ni p\mapsto E_p=T_p\mathfrak{Z}_{\pi(p)}\subseteq T_p(\C\otimes\H\times\sfera)\;.$$
The integral manifolds of this subbundle are obviously the manifolds $\mathfrak{Z}_{\pi(p)}$; on the other hand, $V'=\mathscr{F}^{-1}(\pi^{-1}(V))$ is an open set in $\cU\times\sfera$ and, for every $(z,I)\in V'$, $T_{\mathscr{F}(z,I)}\mathscr{F}(V')\subseteq E_{\mathscr{F}(z,I)}$.
Therefore $\mathscr{F}(V')$ is contained in an integral manifold of the distribution $p\mapsto E_p$, so in $\mathfrak{Z}_{f(q)}$. This implies that $f$ is not injective, as $f_{\vert{\pi(V')}}\equiv f(q)$. 

$$
\begindc{\commdiag}[3]
\obj(0,180)[A]{$\cU\times\sfera$}
\obj(-105,185)[G]{$V'\subset$}
\obj(280,180)[B]{$\C\otimes \H\times\sfera$}
\obj(0,0)[C]{$U$}
\obj(280,0)[D]{$\H$}
\obj(340,0)[G]{$\supset V$}
\obj(280,360)[E]{$E\subset T(\C\otimes \H\times\sfera)$}
\obj(560,180)[F]{$\C\times\sfera$}
\mor{A}{B}{$\mathscr{F}$}[\atleft,\solidarrow]
\mor{A}{C}{$\pi$}[\atright,\solidarrow]
\mor{C}{D}{$f$}[\atright,\solidarrow]
\mor{B}{D}{$\pi$}[\atleft,\solidarrow]
\mor{B}{E}{}
\mor{B}{F}{$(\Phi_\cdot,id)$}
\enddc$$

\end{proof}

\section{Powers in $\C\otimes\H$ as covering maps}\label{secpower}

Having proved Theorem~\ref{teoremacover} it is natural to study the case of $k$-th $\star$-powers, aiming to a result giving necessary and sufficient conditions in order to have the existence of global $k$-th $\star$-roots of
a given slice function. Notice, however, that the $k$-th $\star$-power
corresponds to the pointwise $k$-th power only on slice preserving functions.
Therefore, to better analyse its features it is more convenient to
 move our
attention to its stem function, living in 
$\C\otimes\H$, where such operator corresponds to the pointwise $k$-th power.

Let, hence, $\sigma_k:\C\otimes\H\to\C\otimes\H$ be the function that raises an element of $\C\otimes\H$ to its $k$-th power. We 
recall the identification given in Formula~\eqref{phimap}. In particular, we denote by $e_m$, $m=0,1,2,3$, the elements of the standard basis
$(1\otimes 1,\ 1\otimes i,\ 1\otimes j,\ 1\otimes k)$, respectively.
The corresponding coordinates will be denoted as $z_0,\ z_1,\ z_2,\ z_3$. From now on, whenever working on $\C^4$ we will keep in mind such an identification.

We therefore denote with the same symbol $\sigma_k$ the corresponding map defined on $\C^4$ with values on $\C^4$.
We have
$$\sigma_k(z_0,z_1,z_2,z_3)=(p^k_0(z_0,\vecnorm{z}), z_1p^{k-1}_1(z_0,\vecnorm{z}),z_2p^{k-1}_1(z_0,\vecnorm{z}),z_3p^{k-1}_1(z_0,\vecnorm{z}))$$
where $\vecpart{z}=z_1e_1+z_2e_2+z_3e_3$, $\vecnorm{z}=z_1^2+z_2^2+z_3^2$ and $p_0^k, p_1^{k-1}$ are such that
$$(x+Iy)^k=p^k_0(x,y^2)+Iyp^{k-1}_1(x,y^2)\;.$$

\begin{rem}\label{remAMPA}
Let $f=\mathcal{I}(F):U\to\H$ be a slice regular function such that 
$f=f_0+f_v=f_0+f_1i+f_2j+f_3k$, with $f_{0},f_{1},f_{2},f_{3}$ slice preserving functions
defined on $U$.
 With this representation,
the function $\sigma_k(F)$ corresponds to the $k$-th $\star$-power
$$
f^{\star k}=p_0^k(f_0,f_v^\mathsf{s})+p_1^{k-1}(f_0,f_v^\mathsf{s})(f_1i+f_2j+f_3k)=p_0^k(f_0,f_v^\mathsf{s})+p_1^{k-1}(f_0,f_v^\mathsf{s})f_v\;.
$$
This representation was used in~\cite[Section 7]{AltavillaAMPA} to study algebraic properties of the $k$-th $\star$-power of a slice regular function.
\end{rem}

For the convenience of what follows, let us define the following function $$N(z):=(z_0^2+z_1^2+z_2^2+z_3^2)^\frac{1}{2}\;.$$
Then, we have the following equalities:
\begin{equation}\label{Cheby}
    p_0^k(z_0,\vecnorm{z})=N(z)^kT_k(t)\;,\qquad p_1^{k-1}(z_0,\vecnorm{z})=N(z)^{k-1}U_{k-1}(t)\;,
\end{equation}
where $t=z_0/N(z)$ and $T_k$ and $U_{k-1}$ denote the Chebyshev polynomials of first kind and degree $k$ and of
second kind and degree $k-1$, respectively.

We 
recall now some well known formul\ae   (see~\cite[Chapters 1 and 2]{Mason}). First of all the \textit{Pell equation} states that, for any $n\in\mathbb{N}$,
$$
T_n^2(x)+(1-x^2)U_{n-1}^2(x)=1\;.
$$
Moreover,  the derivation rules of the Chebyshev polynomials read as follows:
$$
\frac{dT_n(x)}{dx}=nU_{n-1}(x),\qquad \frac{dU_{n}}{dx}=\frac{xU_n(x)-(n+1)T_{n+1}(x)}{1-x^2}\;.
$$
Furthermore, recalling that $t=z_0/N(z)$, we have that
$$
\frac{\partial t}{\partial z_0}=\frac{1}{N(z)}(1-t^2),\qquad \frac{\partial t}{\partial z_n}=-t\frac{z_n}{N(z)^2}\;,
$$
for $n=1,2,3$.

We now give an explicit formula for the differential of $\sigma_k$.

\begin{lemma}\label{lmm_jacobianpower}
$\det D\sigma_k(z)$ is a function of $z_0^2$ and $\vecnorm{z}$. More precisely,
$$\det D\sigma_k(z)=k^2(z_0^2+\vecnorm{z})[p_1^{k-1}(z_0,\vecnorm{z})]^2,$$
where $p_1^{k-1}(z_0,\vecnorm{z})=(z_0^2+\vecnorm{z})^{k-1}U_{k-1}\left(\frac{z_0}{(z_0^2+\vecnorm{z})^1/2}\right)$,
and $U_{k-1}$ denotes the Chebyshev polynomial of second kind and of degree $k-1$.
\end{lemma}

\begin{proof}
Assume first that $N(z)\neq 0$.
We start by recalling that 
$$\sigma_k(z_0,z_1,z_2,z_3)=(p^k_0(z_0,\vecnorm{z}), z_1p^{k-1}_1(z_0,\vecnorm{z}),z_2p^{k-1}_1(z_0,\vecnorm{z}),z_3p^{k-1}_1(z_0,\vecnorm{z}))$$ 
and recalling the equalities in Formula~\eqref{Cheby},
we have
$$
\frac{\partial p^k_0(z_0,\vecnorm{z})}{\partial z_0}=\frac{\partial \left(N(z)^kT_k(t)\right)}{\partial z_0}=kN(z)^{k-1}[tT_k(t)+(1-t^2)U_{k-1}(t)]=:A\;,
$$
for $n=1,2,3$, 
$$
\frac{\partial p^k_0(z_0,\vecnorm{z})}{\partial z_n}=\frac{\partial \left(N(z)^kT_k(t)\right)}{\partial z_n}=kN(z)^{k-2}[T_k(t)-tU_{k-1}(t)]z_n=:Bz_n\;.
$$
Passing now to the other components, for $n,m=1,2,3$, we get
$$\frac{\partial z_mp^{k-1}_1(z_0,\vecnorm{z})}{\partial z_0}=\frac{\partial \left(z_m N(z)^{k-1}U_{k-1}(t)\right)}{\partial z_0}=kN(z)^{k-2}[tU_{k-1}(t)-T_k(t)]z_m=:-Bz_m\;,$$
and
\begin{align*}
    \frac{\partial z_mp^{k-1}_1(z_0,\vecnorm{z})}{\partial z_n}&=\frac{\partial \left(z_m N(z)^{k-1}U_{k-1}(t)\right)}{\partial z_n}\\
    &=\delta_{n m}N(z)^{k-1}U_{k-1}(t)+N(z)^{k-3}[(k-1)U_{k-1}(t)+\frac{t}{1-t^2}(kT_k(t)-tU_{k-1}(t))]z_n z_m\\
    &=:\delta_{n m}H+Cz_n z_m\;,
\end{align*}
where $\delta_{n m}$ denotes the \textit{Kronecker delta}.

Thanks to these computations the Jacobian of $\sigma_k$ is equal to
\begin{align*}
   \det D\sigma_k(z)&=\det\begin{pmatrix}
A & Bz_1 & Bz_2 & Bz_3\\
-Bz_1 & H+Cz_1^2 & Cz_1z_2 & Cz_1z_3\\
-Bz_2 & Cz_1z_2 & H+Cz_2^2 & Cz_2z_3\\
-Bz_3 & Cz_1z_3 & Cz_2z_3 & H+Cz_3^2
\end{pmatrix}\\
&=H^2[AH+(AC+B^2)(z_1^2+z_2^2+z_3^2)]\\
&=H^2[AH+(AC+B^2)N(z)^2(1-t^2)]\;.
\end{align*}

We now analyze the term $(AC+B^2)N(z)^2(1-t^2)$. This can be written as follows
\begin{align*}
    (AC+B^2)N(z)^2(1-t^2)=&N(z)^2(1-t^2)\{kN(z)^{k-1}[tT_k(t)+(1-t^2)U_{k-1}(t)]\times\\
    &\times N(z)^{k-3}[(k-1)U_{k-1}(t)+\frac{t}{1-t^2}(kT_k(t)-tU_{k-1}(t))]+\\
    &+{[}kN(z)^{k-2}[T_k(t)-tU_{k-1}(t)]{]}^2\\
    =&kN(z)^{2(k-1)}(1-t^2)\{T_k(t)U_{k-1}(t)[t(k-1)-t\frac{t^2}{1-t^2}+tk-2tk]+\\
    &+T_k^2(t)[k\frac{t^2}{1-t^2}+k]+U_{k-1}^2(t[(1-t^2)(k-1)-t^2+kt^2])\}\\
    =&kN(z)^{2(k-1)}\{kT_k^2(t)+(1-t^2)(k-1)U_{k-1}^2(t)-tT_k(t)U_{k-1}(t)\}\;.
\end{align*}

Therefore we have
\begin{align*}
    \det D\sigma_k(z)=&H^2[AH+(AC+B^2)N(z)^2(1-t^2)]\\
    =&[N(z)^{k-1}U_{k-1}(t)]^2\{kN(z)^{k-1}[tT_k(t)+(1-t^2)U_{k-1}(t)]\times\\
    &\times N(z)^{k-1}U_{k-1}(t)+\\
    &+kN(z)^{2(k-1)}\{kT_k^2(t)+(1-t^2)(k-1)U_{k-1}^2(t)-tT_k(t)U_{k-1}(t)\}\}\\
    =&k^2N(z)^{2(k-1)}[N(z)^{k-1}U_{k-1}(t)]^2\\
    =&k^2(z_0^2+z_1^2+z_2^2+z_3^2)[p_1^{k-1}(z_0,\vecnorm{z})]^2\;,
\end{align*}
where in the third equality we have used the Pell equation.
Hence we have proved the equality whenever $N(z)\neq0$. Now, as the resulting polynomial is of the right degree, then this equality extends also
to the set of points where $N(z)=0$.
\end{proof}

We now want to analyze the zero set of $D\sigma_k$. 
We set 
$$Q^k(t)=\sum_{h=0}^{[k/2]}(-1)^h{k\choose 2h+1}t^{k-1-2h}=\Im((t+\imath)^k)$$
and we notice that $Q^{k}(t)=p_1^{k-1}(t,1)=(t^2+1)^{k-1}U_{k-1}(\frac{t}{1+t^2})$.
Therefore $Q^{k}(t)=0$ if and only if $U_{k-1}(t)=0$. Now, the zeroes of the Chebyshev polynomials are well known (in particular, $U_{k-1}(t)=0$ if and only if $t=\cos(\frac{n}{k}\pi)$ with $n=1,\dots, k-1$), however to be self-contained, we present here a proof of the fact that such roots are all simple. We will later make use of techniques involved in the proof.

\begin{lemma}\label{lmm_Qroots}
$Q^{k}(t)$ has all real and simple roots.
\end{lemma}
\begin{proof}
By definition, for $t\in\mathbb{R}$
$$Q^{k}(t)=\frac{(t+\imath)^k-(t-\imath)^k}{2\imath}$$
so, $Q^{k}(t)=0$ if and only if 
$$\left(\frac{t-\imath}{t+\imath}\right)^k=1$$
i.e. if and only if $(t-\imath)/(t+\imath)$ is a $k$-th root of unity. Now, the map \begin{equation}\label{CayleyMap}
    \mathfrak{C}(\zeta)=\frac{\zeta-\imath}{\zeta+\imath}
\end{equation}
is the Cayley transform which is a biholomorphism between the upper half-plane and the unit disc; in particular, it maps the real line bijectively on the unit circle minus the point $\{1\}$, which contains $k-1$ of the $k$-th roots of unity.
For $k$ odd, $Q^k(t)$ is a polynomial in $t^2$, for $k$ even, $Q^k(t)/t$ is a polynomial in $t^2$; therefore, for $k$ odd, $Q^k(t)$ has $(k-1)/2$ positive roots and $(k-1)/2$ negative roots, for $k$ even, $Q^k(t)$ has a zero root, $(k-2)/2$ positive roots and $(k-2)/2$ negative roots.

Therefore, $Q^k(t)$ has $k-1=\deg Q^k(t)$ real roots, which are then all simple roots, and $[(k-1)/2]$ of them are positive roots.
\end{proof}

We now pass to prove that, under suitable hypotheses on the domain and range, $\sigma_k$ is a covering map. In order to
obtain such a result, we need to prepare convenient notations.

Let us consider the set of imaginary units in $\C\otimes \H$
$$\mathcal{S}=\{s=(0,\vecpart{z})\in\C^4\ :\ \vecnorm{z}=1\}\simeq \{p+\sqrt{-1}q\in\C\otimes\H\ :\ (p+\sqrt{-1}q)^2=-1\}.$$
\begin{rem}
To clarify the definition of $\mathcal{S}$, if 
$p=p_1i+p_2j+p_3k,q=q_1+q_2j+q_3k\in\H$, then $(p+\sqrt{-1}q)^2=p^2-q^2+\sqrt{-1}(pq+qp)=-||p||^2+||q||^2+2\sqrt{-1}\langle p,q\rangle$. Therefore,  $p+\sqrt{-1}q\in\mathcal{S}$ if and only if 
$$
\begin{cases}
-||p||^2+||q||^2=1\\
\langle p,q\rangle=0.
\end{cases}$$

Notice that, if $\Im(z)=0$ and $p\in\sfera$ (which corresponds to $q=0$), then $z\otimes p\in\mathcal{S}$, while if $z\otimes p\in\mathcal{S}$, then
$p\in\sfera$ does not hold in general. In fact, the element $i+j+\sqrt{-1}(\frac{1}{3}i-\frac{1}{3}j+\frac{\sqrt{7}}{3}k)$ belongs to $\mathcal{S}$,
but in general $\pi(i+j+\sqrt{-1}(\frac{1}{3}i-\frac{1}{3}j+\frac{\sqrt{7}}{3}k),I)\notin \sfera$: if $I=i$, then $(i+j+i(\frac{1}{3}i-\frac{1}{3}j+\frac{\sqrt{7}}{3}k))^2\neq-1$.
\end{rem}
In analogy with the map $\pi:\C\times\sfera\to\H$, we define
\begin{equation}\label{rho}
\rho:\C^2\times\mathcal{S}\to\mathbb{C}^4
\end{equation}
given by
$$\rho((u_0,u_1),s)=u_0e_0+u_1s\;.$$
The restriction of $\rho$ to $W'=(\C^2\setminus\{u_1=0\})\times\mathcal{S}$ is a $2$-to-$1$ cover of its image $\Omega'=\rho(W')$, given by
$$\Omega'=\{(z_0,\vecpart{z})\in\C^4\ :\ \vecnorm{z}\neq 0\}\;.$$

Given $(z_0,\vecpart{z})\in \Omega'$ we have that $$\rho^{-1}(z_0,\vecpart{z})=\left\{\left((z_0,\sqrt{\vecpart{z}^2}),\frac{\vecpart{z}}{\sqrt{\vecpart{z}^2}}\right),\,\left((z_0,-\sqrt{\vecpart{z}^2}),-\frac{\vecpart{z}}{\sqrt{\vecpart{z}^2}} \right)\right\}.$$

Moreover, if $(z_0,\vecpart{z})\in\C^4$ and $(w_0,\vecpart{w})=\sigma_k((z_0,\vecpart{z}))$, then
\begin{equation}\label{systemroots}
    \begin{cases}w_0=p_0^k(z_0,\vecnorm{z})\\
\vecpart{w}=p_1^{k-1}(z_0,\vecnorm{z})\vecpart{z}
\end{cases}
\end{equation}
therefore, if we put
$$\Omega'_k=\{(z_0,\vecpart{z})\in\C^4\ :\ \vecnorm{z}\neq 0,\ p_1^{k-1}(z_0,\vecnorm{z})\neq 0\}$$
we have that $\sigma_k(\Omega'_k)\subseteq\Omega'$.

We consider $W_k'=\rho^{-1}(\Omega'_k)\subset\C^2\times\mathcal{S}$ and the map $\mathfrak{s}_k:W'_k\to W'$ such that $\rho\circ \mathfrak{s}_k=\sigma_k\circ\rho$ given by
$$\mathfrak{s}_k((u_0,u_1),s)=((p_0^k(u_0,u_1^2), u_1p_1^k(u_0,u_1^2)), s)\;.$$

$$
\begindc{\commdiag}[3]
\obj(0,180)[A]{$W_k'$}
\obj(280,180)[B]{$W'$}
\obj(0,0)[C]{$\Omega_k'$}
\obj(280,0)[D]{$\Omega'$}
\mor{A}{B}{$\mathfrak{s}_k$}[\atleft,\solidarrow]
\mor{A}{C}{$\rho$}[\atright,\solidarrow]
\mor{C}{D}{$\sigma_k$}[\atright,\solidarrow]
\mor{B}{D}{$\rho$}[\atleft,\solidarrow]
\enddc
$$

Finally, let 
$$\Omega=\{(z_0,\vecpart{z})\in\Omega'\ :\ z_0^2+\vecnorm{z}\neq 0,\ z_0\neq 0\}$$
and set $\Omega_k=\Omega'_k\cap\Omega$.

\begin{rem}
Notice that the problem of finding a solution $(z_0,\vecpart{v})\in\C^4$
for the system in Formula~\eqref{systemroots}, can be thought as finding a slice regular function $f$ such that $f^{\star k}=g$, for a given $g$ slice regular. What follows describes, in our geometric language, the natural procedure that one would exploit in
a purely algebraic setting. In particular, if one assume that $\mathcal{I}_g=g_v/\sqrt{g_v^\mathsf{s}}$ and
$\mathcal{I}_f=f_v/\sqrt{f_v^\mathsf{s}}$
are well defined, then we can write $f=f_0+f_1\mathcal{I}_f$, $g=g_0+g_1\mathcal{I}_g$, with $f_0,f_1,g_0,g_1$ slice preserving, and
$f^{\star k}=p_0^k(f_0,f_1^2)+f_1p_1^{k-1}(f_0,f_1^2)\mathcal{I}_f=g_0+g_1\mathcal{I}_g$.
Starting from these formal computation, it is possible to find a number of suggestions
on how to proceed and what is needed to be proved. As one of the aim of this paper is
to show how some of our geometric infrastructure works, we will not follow
this approach, but it will help the reader keeping it in mind when we treat explicit examples.

\end{rem}

\begin{rem}\label{remomegaslice}
Let $f:U\to \H$ be a slice regular function, then with the notation introduced in~\cites{AltavillaAMPA}, we have that, for any $I\in\sfera$,
\begin{itemize}
    \item $F(z)\in \Omega'$ if and only if $f_v^{\mathsf{s}}(\pi(z,I))\neq0$;
    \item $F(z)\in \Omega_k'$ if and only if $f_v^\mathsf{s}(\pi(z,I))\neq0$ and $p_1^{k-1}(f_0(\pi(z,I)),f_v^\mathsf{s}(\pi(z,I)))\neq 0$ (see~\cite[Section 7]{AltavillaAMPA});
    \item $F(z)\in \Omega$ if and only if $f^\mathsf{s}(\pi(z,I))\neq0$ and $f_0(\pi(z,I))\neq 0$.
\end{itemize}

\end{rem}

\begin{propos}\label{prp_preimages}
$\sigma_k$ is $k^2$-to-$1$ from $\Omega_k$ to $\Omega$.
\end{propos}
\begin{proof}
We consider $W=\rho^{-1}(\Omega)$ and $W_k=\rho^{-1}(\Omega_k)$; if we show that the map $\mathfrak{s}_k$ is $k^2$-to-$1$ from $W_k$ to $W$, the thesis will follow, as $\mathfrak{s}_k$ acts as the identity on $\mathcal{S}$, so the preimages of a point through $\mathfrak{s}_k$ are sent to different values by $\rho$.

Given $((v_0,v_1),s')\in W$, the equation $\mathfrak{s}_k((u_0,u_1),s)=((v_0,v_1),s')$ is satisfied if and only if $s=s'$ and 
$$v_0=p_0^k(u_0,u_1^2)\qquad v_1=u_1p_1^{k-1}(u_0, u_1^2)\;.$$
These two are polynomial equations in the variables $(u_0,u_1)$ of (total) degree $k$, so, by the affine Bezout theorem, the system they form has at most $k^2$ solutions.

We set $t=u_0/u_1$, as in $W_k$ we have that $u_1\neq 0$, then
$$p_0^k(t,1)=\frac{(t+\imath)^k+(t-\imath)^k}{2}\;,\qquad p_1^{k-1}(t,1)=\frac{(t+\imath)^k-(t-\imath)^k}{2\imath}$$
so, as $v_1=u_1p_1^{k-1}(u_0,u_1)=u_1^{k}p^{k-1}_1(t,1)\neq 0$ in $W_k$,
$$\lambda=\frac{v_0}{v_1}=\imath\frac{(t+\imath)^k+(t-\imath)^k}{(t+\imath)^k-(t-\imath)^k}=\imath\frac{1+\mathfrak{C}^k(t)}{1-\mathfrak{C}^k(t)}$$
where $\mathfrak{C}$ is the Cayley transform introduced in Formula~\eqref{CayleyMap}, which is an automorphism of the Riemann sphere, sending $\infty$ to $1$. We have
\begin{equation}\label{eq_cayley}\mathfrak{C}^k(t)=\frac{\lambda-\imath}{\lambda+\imath}=\mathfrak{C}(\lambda)\end{equation}
and we note that $\mathfrak{C}(\lambda)=0,\infty$ if and only if $\lambda=\pm \imath$ if and only if $v_0=\pm \imath v_1$ if and only if $v_0^2+\vecnorm{v}=0$ which does not happen in $W$.

Therefore, for $((v_0,v_1),s)\in W$, we have $k$ distinct solutions $t_1,\ldots, t_k$ to equation \eqref{eq_cayley}. If we fix $m\in \{1,\ldots, k\}$, then $u_0=t_mu_1$ and
$$v_1=u_1^kp_1^{k-1}(t_m,1)\;.$$

Hence, we obtain 
$$u_1^k=\frac{v_1}{p_1^{k-1}(t_m,1)}$$
which has $k$ different solutions $\omega_{m,1},\ldots, \omega_{m,k}$ for each $m\in \{1,\ldots, k\}$. This gives a total of $k^2$ solutions of the form $(t_m\omega_{m,n}, \omega_{m,n})$ for $m,n\in\{1,\ldots,k\}$.
\end{proof}

We are now able to state our result.

\begin{teorema}\label{teo_powercover}
The function $\sigma_k:\Omega_k\to\Omega$ is a covering map of degree $k^2$.
\end{teorema}
\begin{proof}
By Lemma \ref{lmm_jacobianpower}, $\sigma_k$ is a local biholomorphism from $\Omega_k$ to $\Omega$ and by Proposition \ref{prp_preimages} it has a constant number of preimages. Therefore it is a covering map.
\end{proof}

In the last part of this section we give a brief description of the action of $\sigma_{k}$
outside $\Omega_{k}$.
We define $X_k=\C^4\setminus \Omega_k$ and $X=\C^4\setminus\Omega$. For $\alpha\in \C$ we set
$$V_\alpha=\{(z_0,\vecpart{z})\in\C^4\ :\ z_0^2=\alpha\vecnorm{z}\}$$
and
$$V_\infty=\{(z_0,\vecpart{z})\in\C^4\ :\ \vecnorm{z}=0\}\;.$$
Notice that if a slice function $f=\mathcal{I}(F)$ is such that $F$ takes values
in $V_{-1}$, then $f^{\mathsf{s}}\equiv 0$, i.e. $f$ is identically zero or a zero divisor with respect to the $\star$-product (see~\cite[Section 2.4]{AltavillaLAA}).

Collecting everything, if $R_k$ denotes the set of non-negative roots of $Q^k(t)$, then
$$X_k=V_{-1}\cup V_0\cup V_\infty\cup\bigcup_{r\in R_k}V_{r^2}\quad
\textrm{and}\quad X=V_{-1}\cup V_0\cup V_{\infty}\;.$$


We now describe the action of $\sigma_k$ where it is not a cover. 

\begin{propos}\label{propsing}
Let $(z_0,\vecpart{z})\in\C^{4}$. We have the following relations:
\begin{enumerate}
\item if $(z_{0},\vecpart{z})\in V_{-1}$, then $\sigma_{k}(z_{0},\vecpart{z})\in V_{-1}$;
\item if $(z_{0},\vecpart{z})\in V_{0}$ and $k$ is odd, then $\sigma_{k}(z_{0},\vecpart{z})\in V_{0}$;
 \item if $(z_{0},\vecpart{z})\in V_{0}$ and $k$ is even, then $\sigma_{k}(z_{0},\vecpart{z})\in V_{\infty}$;
 \item  if $(z_{0},\vecpart{z})\in V_{\infty}$, then $\sigma_{k}(z_{0},\vecpart{z})\in V_{\infty}$;
 \item  if $(z_{0},\vecpart{z})\in V_{r^{2}}$, with $r\in R_{k}$, then $\vecpart{\sigma_{k}(z_{0},\vecpart{z})}=0$ and hence $\sigma_{k}(z_{0},\vecpart{z})\in V_{\infty}$.
\end{enumerate}
\end{propos}

\begin{proof}
For point $(1)$ it sufficient to notice that if $z_{0}^{2}+\vecpart{z}^{2}=0$, then
$\sigma_{2}(z_{0},\vecpart{z})=(z_{0}^{2}-\vecpart{z}^{2},2z_{0}\vecpart{z})$ and
$(z_{0}^{2}-\vecpart{z}^{2})^{2}+4z_{0}^{2}\vecpart{z}^{2}=(z_{0}^{2}+\vecpart{z}^{2})^{2}=0$.

For points $(2)$ and $(3)$ it is sufficient to recall symmetry properties for Chebyshev polynomials:
$$
T_{n}(-x)=(-1)^{n}T_{n}(x)\;,\qquad U_{n}(-x)=(-1)^{n}U_{n}(x)\;.
$$  

We pass now to point $(4)$: in this case, as $\vecpart{z}^{2}=0$, then
$$
\sigma_{k}(z_{0},\vecpart{z})=(z_{0}^{k}T_{k}(1),\vecpart{z}z^{k-1}U_{k-1}(1))=(z_{0}^{k},\vecpart{z}z^{k-1}U_{k-1}(1))\;,
$$
and hence $\vecpart{\sigma_{k}(z_{0},\vecpart{z})}^{2}=0$.

The last point $(5)$ is trivial as $R_{k}$ is the set of nonnegative roots of $Q^{k}(t)=(t^{2}+1)^{k-1}U_{k-1}(\frac{t}{1+t^{2}})$.
\end{proof}
\begin{rem}
Notice that, if $(z_0,\vecpart{z})\in V_{-1}$, then 
$$p_0^k(z_0, \vecnorm{z})=z_0^k\qquad p_1^{k-1}(z_0,\vecnorm{z})=z_0^{k-1}\;.$$
Therefore, $\sigma_k(z_0,\vecpart{z})=z_0^{k-1}(z_0,\vecpart{z})$ is a $k$-to-$1$ covering of $V_{-1}$ on itself.
The same happens with $V_{\infty}$.
\end{rem}

In the next corollary we interpret Proposition~\ref{propsing} in terms of slice regular functions.
\begin{corol}
Let $f=f_{0}+f_{1}i+f_{2}j+f_{3}k=f_{0}+f_{v}:U\to \H$ be a slice function. Then we have the following relations:
\begin{enumerate}
\item if $f^{\mathsf{s}}\equiv0$, then $(f^{\star k})^{\mathsf{s}}\equiv 0$;
\item if $f_{0}\equiv 0$ and $k$ is odd, then $(f^{\star k})_{0}\equiv 0$;
\item if $f_{0}\equiv 0$ and $k$ is even, then $(f^{\star k})_{v}^{\mathsf{s}}\equiv 0$;
\item if $f_{v}^{\mathsf{s}}\equiv 0$, then $(f^{\star k})_{v}^{\mathsf{s}}\equiv 0$;
\item if $f_{0}^{2}=r^{2}f_{v}^{\mathsf{s}}$, with $r\in R_{k}$, then $(f^{\star k})_{v}\equiv 0$
and hence  $(f^{\star k})_{v}^{\mathsf{s}}\equiv 0$.
\end{enumerate}
\end{corol}

The proof of this last corollary is a direct application of Proposition~\ref{propsing} taking into account
Remark~\ref{remAMPA}, but indeed many of 
statements can be easily deduced from Proposition~\ref{propsing}.
However, notice that point $(5)$ confirms~\cite[Proposition 7.4]{AltavillaAMPA}.


\section{Global $\star$-roots of a slice function}\label{starrootsec}

In this section we apply the results obtained at the end of the last section for the $k$-th power in $\C\otimes\H$, to obtain
results for the $k$-th $\star$-root of a slice function and at the end we also provide a couple of
explicit examples.
We recall Assumption~\ref{assumption1} and we add the following.

\begin{assumption}
From now on, the set $\cU\subset\C$ will always be a simply connected open domain such that $\overline{\cU}=\cU$ or the union of two simply connected open domains that are symmetric with respect to the real line.
\end{assumption}

This last assumption is consistent with the one adopted in~\cite{altavillaLOG} and in~\cite{GPV2} where
the authors call the resulting set $U$ \textit{basic domain}. Notice, in particular, that the hypothesis of simply connectedness is unavoidable, due to standard theory of lifting maps to a covering space (see e.g.~\cite{Kosniowski})

Let $f:U\to\H$ be a slice function.
 We denote by $F:\mathcal{U}\to\C\otimes\H$ its stem function. As in the previous sections, with a small abuse of notation, we will identify $\C\otimes\H$ with $\C^{4}$. So, we will denote with the same
symbol the function $F$ when its range is $\C^{4}$ instead of $\C\otimes\H$ (we
recall that this identification is explicitly given in Formula~\eqref{phimap}).

\begin{propos}\label{prplifts}
If $F(\mathcal{U})\subseteq\Omega$, then there exist $k^2$ functions $F_1,\ldots, F_{k^2}:\mathcal{U}\to\C\otimes\H$ such that $\sigma_k\circ F_m=F$ for $m=1,\ldots, k^2\;.$
\end{propos}
\begin{proof}
This is a direct consequence of Theorem \ref{teo_powercover} and the lifting property of covering maps.
\end{proof}

We notice that, by the very definition of $\sigma_k$, if $F$ and $G$ are stem functions defining two slice functions $f$ and $g$, respectively, then 
$$F=\sigma_k\circ G\Leftrightarrow f=(g)^{\star k}\;.$$

\begin{rem}\label{remarkslice}
As $F$ is a stem function, then $F(\overline{z})=\overline{F(z)}$. Moreover, by the definition of $\sigma$, we have $\sigma_k(\overline{z})=\overline{\sigma_k(z)}$. It follows that, if $\sigma_k\circ G= F$, then 
$$\sigma_k(G(\overline{z}))=\overline{\sigma_k(G(z))}=\sigma_k(\overline{G(z)})\;.$$

Therefore, if $G$ is a solution of the equation $\sigma_k(G)=F$ on the domain $\mathcal{U}$, then the function
$\widehat{G}(z)=\overline{G(\bar z)}$ is another solution on the same domain.
\end{rem}

Thanks to the previous remark we are able to prove the following unexpected result.

\begin{teorema}\label{casosenzareali}
If $F(\mathcal{U})\subset\Omega$ and $\mathcal{U}\cap\R=\emptyset$, then there exist $k^2$ slice functions $f_m:U\to\H$, $m=1,\dots,k^2$, such that
$$(f_m)^{\star k}=f\;.$$
\end{teorema}
\begin{proof}
From Proposition~\ref{prplifts} we now that there are $k^2$ functions $G_1,\ldots, G_{k^2}:\mathcal{U}\to\C\otimes\H$ such that $\sigma_k\circ G_m=F$.
Now, from Remark~\ref{remarkslice} for any $m=1,\dots, k^2$ there exists $n=1,\dots, k^2$, such that $\widehat G_m(z)=G_n(z)$, for all $z\in \mathcal{U}$.
Moreover, such $n$ is unique from Theorem~\ref{teo_powercover}. Therefore, we can define a function $\tau:\{1,\dots,k^2\}\to\{1,\dots, k^2\}$ as
$\tau(m)=n$, where $n$ is such that $\widehat G_m(z)=G_n(z)$ for all $z\in\mathcal{U}$.
Hence, we are able to define the following stem functions $F_\mu:\mathcal{U}\to\C\otimes\H$ as
$$
F_\mu(z)=\begin{cases}
G_\mu(z),&\mbox{if }z\in\mathcal{U}^+\;,\\
G_{\tau(\mu)}(z),&\mbox{if }z\in\mathcal{U}^-.
\end{cases}$$
Notice that all these stem functions are well defined as $\cU\cap\R=\emptyset$.
Then, the induced stem functions $f_1,\dots,f_{k^2}$ are such that $(f_\mu)^{\star k}=f$ for any $\mu=1,\dots, k^2$.
\end{proof}

The case in which the domain of $f$ has real points follows the general expectation. However, imposing such a condition on the 
domain implies greater effort in the proof: the stem function of any $k$-th root of $f$ has to be real on $\mathcal{U}\cap\R$.

\begin{rem}Let $\R^4$ denote the set of real points in $\C^4$, corresponding to the set $\{p+\sqrt{-1}\cdot 0\ :\ p\in\H\}\subset\C\otimes\H$, then, on it the projection $\pi:\C\otimes\H\times\sfera\to\H$ restricts to  the map $h:\R^4\to\H$ given by $h(x_0,x_1,x_2,x_3)=x_0+x_1i+x_2j+x_3k$. We notice that $\sigma_k(\R^4)\subseteq\R^4$ for all $k$ and 
\begin{equation}\label{eqovvia}h\circ ({\sigma_k}_{|\R^4})=h^k\;,\end{equation}
where in the right hand side we mean the $k$-th power in the algebra $\H$.\end{rem}

In view of these considerations, if $G$ and $F$ are two stem functions such that $\sigma_k\circ G=F$, then, given $x\in \mathcal{U}\cap\R$, $F(x)\in\R^4$ and $G(x)\in\R^4$ are such that $h(G(x))^k=h(F(x))$ (the power is understood as an operation in $\H$).
As we already know, if $F(x)$ is outside a critical set, the equation $q^k=h(F(x))$ has exactly $k$ solutions in $\H$, which correspond to $k$ points in $\R^4$ via $h$.

\begin{teorema}\label{teo_kroots}
If $F(\mathcal{U})\subseteq\Omega$ and $\mathcal{U}\cap\R\neq\emptyset$, then there exist $k$ slice functions $f_m:U\to\H$, $m=1,\dots,k$, such that
$$(f_m)^{\star k}=f\;.$$
\end{teorema}
\begin{proof}
As we know from Theorem \ref{teoremacover}, the map $g(q)=q^k$ is a covering map from $\H\setminus S(g)$ to $\H\setminus C(g)$. An easy computation shows that $\H\setminus C(g)=h(\R^4\cap\Omega)$ and $\H\setminus S(g)=h(\R^4\cap\Omega_k)$, where $\Omega$ and $\Omega_k$ are the sets defined in the previous section.

Let now $x^0$ be a point in $\mathcal{U}\cap\R$, then $F(x^0)\in\R^4$. Consider $q_0=h(F(x^0))$. Then, by hypothesis, $q_0\in \H\setminus C(g)$, so there exist $k$ preimages of $q_0$ via $g$ in $\H\setminus S(g)$. Denoting them by $q_1,\ldots, q_k$, by \eqref{eqovvia}, we have $\sigma_k(h^{-1}(q_m))=F(x^0)$.

We consider the lift $F_m$ of $F$ via $\sigma_k$ such that $F_m(x^0)=h^{-1}(q_m)$. We will show that $F_m$ is a stem function.

Consider $G(z)=\overline{F_m(\bar{z})}$. As $\sigma_k(\bar z)=\overline{\sigma_k(z)}$, we have
$$\sigma_k(G(z))=\sigma_k(\overline{F_m(\bar{z})})=\overline{\sigma_k(F_m(\bar{z}))}\;.$$
Now, $$\sigma_k(F_m(\bar{z}))=F(\bar{z})=\overline{F(z)}$$
because $F$ is a stem function; so
$$\sigma_k(G(z))=F(z)$$
hence $G(z)$ is one of the $k^2$ lifts of $F$ via $\sigma_k$. However, $G(x^0)=h^{-1}(q_m)$, so $G$ must coincide with $F_m$. So
$$F_m(\bar{x})=\overline{F_m(z)}$$
which means that $F_m$ is a stem function. Let $f_m$ be the corresponding slice function, then, by construction $(f_m)^{\star k}=f$.
\end{proof}

As an obvious corollary, we get the corresponding statement for slice regular functions, as the map $\sigma_k$ is holomorphic.

\begin{corol}
A slice regular function $f$ satisfying the hypothesis of the previous theorem has $k$ slice regular $\star$-kth roots.
\end{corol}

In the hypothesis of the previous theorem, given $F_m$ as above and given $\xi\in\C$ a primitive $k$-th rooth of unity, then 
$F_{m,a}$
$=\xi^a F_m$ is again a lift of $F$: $\sigma_k$ is $k$-homogeneous, so
$$\sigma_k(\xi^a F_m(z))=\xi^{ak}\sigma_k(F_m(z))=\sigma_k(F_m(z))=F(z)\;.$$
If $F_{m,a}(z)=F_{n,b}(z)$, then 
$$\xi^aF_m(z)=\xi^bF_n(z)$$
again, for $z=x^0$, we obtain that $F_m(x^0),\ F_n(x^0)\in\R^4$, therefore $\xi^{a-b}\in\R$. This is possible if $a=b$ or if $2(a-b)=k$ (and $k$ is even); in the first case, it follows that also $n=m$. So, for $k$ odd, the $k^2$ lifts given by Proposition~\ref{prplifts} are precisely the functions $F_{m,a}$, with
$F_m$ the lift of $F$ via $\sigma_k$ such that $F_m(x^0)=h^{-1}(q_m)$.

Moreover,
$$\overline{F_{m,a}(z)}=\overline{\xi^aF_m(z)}=\xi^{-a}\overline{F_m(z)}=\xi^{-a}F_m(\overline{z})=\xi^{-2a}F_{m,a}(\overline{z})\;.$$
Again, if $k$ is odd, the only possibility is $a=0$, so that we have $k$ families of $k$ functions each such that $G(\bar{z})=\xi^c\overline{G(z)}$ and each family has a different value of $c\in\{0,\ldots, k-1\}$.

We now want to show a couple of explicit examples suggesting the general form of such $k$-th $\star$-roots.
For the convenience of what follows we need to introduce a special slice regular function.
All the previous considerations will be deepen in the next section where some algebraic tool will be developed in order to treat the monodromy of
the $\star$-roots. At this stage we only add a couple of examples showing the different behaviours of the odd and even cases. Let us
begin with the following definition.

\begin{defin}\label{defcJ}
Let $U\subset\H$ be an open domain such that $U\cap\R=\emptyset$. We define the slice regular function $\mathcal{J}:U\to\H$ as
$J(q)=\frac{q_v}{||q_v||}$.
\end{defin}
The function $\cJ$ is slice preserving and it is induced by the function 
$$
J(z)=\begin{cases}
\sqrt{-1},& z\in\cU^+,\\
-\sqrt{-1},& z\in\cU^-,
\end{cases}
$$
corresponding to the locally constant complex curve
$$
z\mapsto
\begin{cases}
(\imath,0,0,0),& z\in\cU^+,\\
(-\imath,0,0,0),& z\in\cU^-.
\end{cases}
$$

From the function $\cJ$ it is possible to construct idempotent functions and zero-divisors for the $\star$-product. The prototypes of which are
the functions $\ell_+,\ell_-:U\to\H$ defined as 
$$
\ell_+(q)=\frac{1-\cJ(q)i}{2}\;,\qquad \ell_-(q)=\frac{1+\cJ(q)i}{2}\;.
$$
These two functions are such that $\ell_+\star\ell_+=\ell_+$ and $\ell_+\star\ell_-\equiv0$.

Thanks to the so called \emph{Peirce decomposition}, any slice regular function $f$ defined on a domain without real points $U$ can be written
as $f=f_+\star\ell_++f_-\star\ell_-$, where $f_+,f_-:U\to\H$ are slice regular functions and not zero divisors.
The case in which both $f_+$ and $f_-$ are regular polynomials is studied in some details in~\cites{AltavillaMATHZ} and corresponds
to the family of slice polynomial functions.

\begin{rem}
Let $\alpha+i\beta\in\C$, be a $k$-th rooth of $1$, then, clearly, for any $I\in\mathbb{S}$, we have that $(\alpha+I\beta)^{\star k}=(\alpha+I\beta)^{ k}=1$. Analogously, for any $q\in\H\setminus\R$ we have the following equality
$$
(\alpha+\cJ(q)\beta)^{\star k}=(\alpha+\cJ(q)\beta)^{k}\equiv 1\;.
$$
\end{rem}

\begin{ex}\label{exdeg3}
Let $\cU$ be a simply connected domain in $\C$ and
consider the regular polynomial function $f:U\to \H$ defined as
$f(q)=q^3+3q^2i-3q-i=(q+i)^{\star 3}$. This function preserves the slice $\C_i$, meaning that $f(U\cap \C_i)\subset \C_i$.
It corresponds to the complex curve $F:\cU\to\C^4$ given by $F(z)=(z^3-3z,3z^2-1,0,0)$. A trivial $3$-rd $\star$-root of
$f$ is given by $g:U\to\H$ defined as $g_0(q)=q+i$ corresponding to the complex curve $G_0:\cU\to\H$ defined as
$G_0(z)=(z,1,0,0)$. Moreover, if $\eta_1=-\frac{1}{2}+\frac{\sqrt{3}}{2}i,\eta_2=-\frac{1}{2}-\frac{\sqrt{3}}{2}i$ are the two non-trivial cubic root of
the unity in $\C_i$, then $g_1=g_0\eta_1$ and $g_2=g_0\eta_2$ are other cubic $\star$-roots of $f$,
corresponding to the complex curves $G_1(z)=(-z/2-\sqrt{3}/2,q\sqrt{3}/2-1/2,0,0)$ and $G_2(z)=(-z/2+\sqrt{3}/2,-q\sqrt{3}/2-1/2,0,0)$, respectively. 
If $U\cap \R=\emptyset$ it is 
possible to compute the remaining roots as follows

$$
\begin{matrix}
g_{0,1}(q)=\left(-\frac{1}{2}+\mathcal{J}(q)\frac{\sqrt{3}}{2}\right)g_0(q)\;, & & g_{0,2}(q)=\left(-\frac{1}{2}-\mathcal{J}(q)\frac{\sqrt{3}}{2}\right)g_0(q)\;,\\
g_{1,1}(q)=\left(-\frac{1}{2}+\mathcal{J}(q)\frac{\sqrt{3}}{2}\right)g_1(q)\;, & & g_{1,2}(q)=\left(-\frac{1}{2}-\mathcal{J}(q)\frac{\sqrt{3}}{2}\right)g_1(q)\;,\\
g_{2,1}(q)=\left(-\frac{1}{2}+\mathcal{J}(q)\frac{\sqrt{3}}{2}\right)g_2(q)\;, & & g_{2,2}(q)=\left(-\frac{1}{2}-\mathcal{J}(q)\frac{\sqrt{3}}{2}\right)g_2(q)\;,
\end{matrix}
$$
corresponding to the complex curves
$$
\begin{matrix}
G_{0,1}(z)=\begin{cases}
\eta_1 G_0(z)& z\in\cU^+\\
\eta_2 G_0(z)& z\in\cU^-
\end{cases}\;, & & G_{0,2}(z)=\begin{cases}
\eta_2 G_0(z)& z\in\cU^+\\
\eta_1 G_0(z)& z\in\cU^-
\end{cases}\;,\\
G_{1,1}(z)=\begin{cases}
\eta_1 G_1(z)& z\in\cU^+\\
\eta_2 G_1(z)& z\in\cU^-
\end{cases}\;, & & G_{1,2}(z)=\begin{cases}
\eta_2 G_1(z)& z\in\cU^+\\
\eta_1 G_1(z)& z\in\cU^-
\end{cases}\;,\\
G_{2,1}(z)=\begin{cases}
\eta_1 G_2(z)& z\in\cU^+\\
\eta_2 G_2(z)& z\in\cU^-
\end{cases}\;, & & G_{2,2}(z)=\begin{cases}
\eta_2 G_2(z)& z\in\cU^+\\
\eta_1 G_2(z)& z\in\cU^-
\end{cases}\;,
\end{matrix}
$$
respectively.
Notice that all these functions are slice polynomial functions. For instance we can write the first one as
$$
g_{0,1}(q)=g_0(q)\eta_1\star\ell_++g_0(q)\eta_2\star\ell_-\;.
$$
\end{ex}

The case $k$ even is more involved. For the moment we only give an example of what happens for $k=2$. In the next section we will try to 
explain this phenomenon from an algebraic point of view. 
\begin{ex}\label{exsubdolo}
Let us consider the function $F:\C\to\C^4(\simeq \C\otimes\H)$ defined by $F(z)=(z^2,-z,-z,1)$. This function intersects $V_\infty$ for $z=\pm \imath/\sqrt{2}$ (it corresponds to the slice regular function $f(q)=(q-i)\star (q-j)=q^2-q(i+j)+k$). We have four solutions of the equation $\sigma_2(G)=F$ around $z=\imath$, namely
$$G_1(z)=\frac{\sqrt{2}}{2}(\imath,\imath z,\imath z,-\imath)$$
$$G_2(z)=\frac{\sqrt{2}}{2}(-\imath,-\imath z,-\imath z,\imath)$$
and, fixing a determination of the square root of $2z^2+1$ on some open set,
$$G_3(z)=\left(\frac{\sqrt{4z^2+2}}{2},\frac{-z}{\sqrt{4z^2+2}}, \frac{-z}{\sqrt{4z^2+2}}, \frac{1}{\sqrt{4z^2+2}}\right)$$
$$G_4(z)=\left(-\frac{\sqrt{4z^2+2}}{2},\frac{z}{\sqrt{4z^2+2}}, \frac{z}{\sqrt{4z^2+2}}, -\frac{1}{\sqrt{4z^2+2}}\right)\;.$$
These lifts ``collide'' when $z$ approaches $\pm \imath$: if we pick the square root such that $\sqrt{-1/2}=\imath/\sqrt{2}$, then $G_1(\imath)=G_3(\imath)$ and $G_2(\imath)=G_4(\imath)$. 
Indeed, for $z=\pm \imath/\sqrt{2}$, $F(\pm \imath/\sqrt{2})\notin \Omega'$, where we do not have $4$ square roots, but only $2$.

Notice that since $G_m(\overline{z})=\overline{G_m(z)}$, for $m=3,4$, then $G_3$ and $G_4$ correspond to stem functions 
but they are not defined on $\C$: we can extend $G_3$ and $G_4$ to any simply connected open domain $\mathcal{V}$ which does not contain $\pm \imath/\sqrt{2}$.
If $V=\pi(\mathcal{V})$, then the slice regular functions $g_3:V\to\H$ and $g_4:V\to \H$ induced by $G_3$ and $G_4$ can be written as
$$
g_3(q)=\sqrt{q^2+\frac{1}{2}}-\frac{q}{\sqrt{4q^2+2}}(i+j)+\frac{k}{\sqrt{4q^2+2}},\quad g_4(q)=-g_3(q),
$$
respectively, where the existence of the square root of $4q^2+2$ is implicitly 
implied by our assumptions.

The phenomenon of reduced solutions (from $4$ to $2$) is not due to the fact that some preimages collide when $z$ approaches $\pm \imath/\sqrt{2}$ (this is what happens at the points where $D\sigma_k$ is not a local diffeomorphism), but rather to the fact that some preimages go to infinity, ``disappearing'', i.e. the limit of $G_3$ and $G_4$, for $z\to\pm\imath\sqrt{2}$, goes to infinity . This is an indication of the fact that $\sigma_k$ is not proper from $\C^4\to\C^4$, but from $\C^4\setminus \Omega'$ to $\C^4\setminus \Omega'$.

Lastly, as in the previous example,
notice that $G_1(\bar z)=\overline{G_2(z)}$, hence, if $\mathcal{V}$ has no real points, then
it is possible to construct the other two square roots of $f$ as explained in the proof of Theorem~\ref{casosenzareali}. 
In particular, the resulting slice regular functions are slice polynomial functions:
$$g_1(q)=-R(q)\star\ell_+ +R(q)\star \ell_-\;,\quad g_2(q)=-g_1(q)\;,$$
where $R(q)=\frac{\sqrt{2}}{2}[q(1+k)-i+j]$.
\end{ex}

\section{Covering automorphisms and monodromy}\label{secmonodromy}

In this section we reinterpret the results given in the previous one in terms of covering automorphisms, obtaining
more refined statements and the monodromy of the maps $\mathfrak{s}_k$ and $\sigma_k$. This will allow us to give a more precise description of what happen in explicit cases.

Let us begin by considering the following commutative diagram
$$
\begindc{\commdiag}[3]
\obj(0,180)[A]{$W_k$}
\obj(280,180)[B]{$W$}
\obj(0,0)[C]{$\Omega_k$}
\obj(280,0)[D]{$\Omega$}
\mor{A}{B}{$\mathfrak{s}_k$}[\atleft,\solidarrow]
\mor{A}{C}{$\rho$}[\atright,\solidarrow]
\mor{C}{D}{$\sigma_k$}[\atright,\solidarrow]
\mor{B}{D}{$\rho$}[\atleft,\solidarrow]
\enddc
$$
where $\rho$ is the map defined in Formula~\eqref{rho} and each arrow is a covering map: more precisely, $\rho$ is a $2$-to-$1$ covering map and $\sigma_k$, $\mathfrak{s}_k$ are $k^2$-to-$1$ covering maps.

Given a covering map $\pi:X\to Y$, we will denote by $\Aut_\pi$ the group of automorphisms $f$ of $X$ such that $\pi\circ f=\pi$.

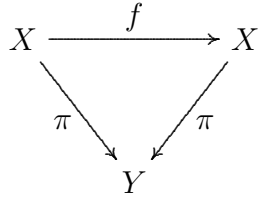
$$
\begindc{\commdiag}[3]
\obj(0,180)[A]{$X$}
\obj(280,180)[B]{$X$}
\obj(140,0)[C]{$Y$}
\mor{A}{B}{$f$}[\atleft,\solidarrow]
\mor{A}{C}{$\pi$}[\atright,\solidarrow]
\mor{B}{C}{$\pi$}[\atleft,\solidarrow]
\enddc
$$

Let us now define the following function.
\begin{defin}

Let $\Gamma:\C^2\times\mathcal{S}\to\C^2\times\mathcal{S}$ be given by $$\Gamma((u_0,u_1),s)=((u_0,-u_1),-s)$$.
\end{defin}

\begin{rem}\label{gamma}
We have that
$$\Aut_\rho=\{1,\ \Gamma\}\;,$$
where we denoted by $1$ the identity automorphism.
\end{rem}

We define two actions of $\C$ on $\C^2$, as the multiplication by the following two matrices
$$z\mapsto z\mathrm{Id}\qquad z\mapsto A_z$$
where $z\in\C$, $\mathrm{Id}$ is the $2\times 2$ identity matrix and
$$A_z=\begin{pmatrix}\Re z&-\Im z\\\Im z&\Re z\end{pmatrix}\,.$$
We denote by $U_k$ the group of complex $k$-th roots of unity. From now on we will only consider maps from $\C^2\times\mathcal{S}$ to itself of the form $((u_0,u_1),s)\mapsto((g_0(u_0,u_1), g_1(u_0,u_1)), s)$, so we will consistently omit the component relative to $\mathcal{S}$.

\begin{rem}
Let $f:U\to\H$ be a slice regular function. Then if $f=f_0+f_1i+f_2j+f_3k=f_0+f_v$,
we have that the corresponding complex curve is given by
$F(z)=(f_0(z),f_1(z),f_2(z),f_3(z))=(f_0,\vecpart{f_v})$. Assume that $F(\cU)\subset\Omega$, then its lifts via $\rho$ is given by
$$\left((f_0,\sqrt{f_v^\mathsf{s}}),\frac{f_v}{\sqrt{f_v^\mathsf{s}}}\right)\in\C^2\times\mathcal{S}\;.$$
\end{rem}

In the next proposition we compute the monodromy of $\mathfrak{s}_k$ in the case in which $k$ is odd.
\begin{propos}\label{propautodd}
Let $k$ be odd. Then
$$\Aut_{\mathfrak{s}_k}=\{\xi A_\eta\ :\ \xi,\eta\in U_k\}\;,$$
which is isomorphic to the group $\mathbb{Z}_k\times\mathbb{Z}_k$.
\end{propos}
\begin{proof}
The map $(\xi,\eta)\mapsto \xi A_\eta$ is a group homomorphism from $U_k\times U_k$ to $M_{2,2}(\C)$, indeed,
$$(\xi A_\eta)(\xi' A_{\eta'})=(\xi\xi')A_{\eta\eta'}\;.$$
Now, its kernel is given by the condition $\xi A_\eta=\mathrm{Id}$, which implies (taking the determinant)
$$\xi^2|\eta|^2=1$$
i.e. $\xi^2=1$, i.e. $\xi=\pm 1$, but as $k$ is odd, then $\xi=1$. 
Moreover, taking the trace, we obtain $2\Re\eta=2$, i.e. $\Re\eta=1$, which implies $\eta=1$. Hence the kernel is trivial and the map is injective.

Finally, as $\mathfrak{s}_k$ is $k$-homogeneous in the $\C^2$ component, it is obvious that $\xi \mathrm{Id}\in\Aut_{\mathfrak{s}_k}$ for all $\xi\in U_k$; on the other hand, 
$$\mathfrak{s}_k((u_0,u_1),s)=((p_0^k(u_0,u_1^2), u_1p_1^{k-1}(u_0,u_1^2)),s)$$
and by definition $(x+\sqrt{-1}y)^k=p_0^k(x,y^2)+\sqrt{-1}yp_1^{k-1}(x,y^2)$, so, from the fact that 
\begin{align*}
    ((\Re(\eta)x-\Im(\eta)y)+\sqrt{-1}(\Im(\eta)x+\Re(\eta)y))^k&=((\Re(\eta)+\sqrt{-1}\Im(\eta))(x+\sqrt{-1}y))^k\\
    &=((x+\sqrt{-1}y))^k
\end{align*}
for $\eta\in U_k$, it follows that $A_\eta\in\Aut_{\mathfrak{s}_k}$.

Given that $U_k\times U_k$ is generated by $(\xi,1)$ and $(1,\eta)$, we obtain that $(\xi,\eta)\to \xi A_\eta$ is an injective homomorphism from $U_k\times U_k$ to $\Aut_{\mathfrak{s}_k}$. However, the latter has at most $k^2$ elements (as $\mathfrak{s}_k$ has degree $k^2$), so the map is an isomorphism.
\end{proof}

We now give an abstract example generalizing Example~\ref{exdeg3}. Before providing it, we recall from~\cite[Proposition 2.10]{AltavillaPAMS} that two slice regular functions $f=f_0+f_v, g=g_0+g_v$ commute with respect to the $\star$-product if and only if there exist two slice preserving functions $\alpha$ and $\beta$ not both identically zero such that $\alpha f_v+\beta g_v\equiv 0$.

\begin{ex}
Let $\cU$ be a simply connected domain such that $\cU\cap\R=\emptyset$
and $f:U\to\H$
be a slice regular function such that $f=f_0+f_v$.
Then clearly $f$ is a cubic $\star$-root of $f^{\star 3}$. By following the procedure of Example~\ref{exdeg3}
and the presentation of $\Aut_{\mathfrak{s}_k}$ given in Proposition~\ref{propautodd} we compute the other $8$ 
roots as follows.
$$
\begin{array}{rrr}
f\;, & f\star\left(-\frac{1}{2}+\frac{\sqrt{3}}{2}\frac{f_v}{\sqrt{f_v^s}}\right)\;, & f\star\left(-\frac{1}{2}-\frac{\sqrt{3}}{2}\frac{f_v}{\sqrt{f_v^s}}\right)\;,\\
\left(-\frac{1}{2}+\cJ\frac{\sqrt{3}}{2}\right)f\;, & \left(-\frac{1}{2}+\cJ\frac{\sqrt{3}}{2}\right)f\star\left(-\frac{1}{2}+\frac{\sqrt{3}}{2}\frac{f_v}{\sqrt{f_v^s}}\right)\;, & \left(-\frac{1}{2}+\cJ\frac{\sqrt{3}}{2}\right)f\star\left(-\frac{1}{2}-\frac{\sqrt{3}}{2}\frac{f_v}{\sqrt{f_v^s}}\right)\;,\\
\left(-\frac{1}{2}-\cJ\frac{\sqrt{3}}{2}\right)f\;, & \left(-\frac{1}{2}-\cJ\frac{\sqrt{3}}{2}\right)f\star\left(-\frac{1}{2}+\frac{\sqrt{3}}{2}\frac{f_v}{\sqrt{f_v^s}}\right)\;, & \left(-\frac{1}{2}-\cJ\frac{\sqrt{3}}{2}\right)f\star\left(-\frac{1}{2}-\frac{\sqrt{3}}{2}\frac{f_v}{\sqrt{f_v^s}}\right)\;,\\
\end{array}
$$
where $\cJ=\cJ(q)$ is the slice regular function given in Definition~\ref{defcJ}.
Notice that, thanks to the previous consideration on the commutativity of the $\star$-product, the factors in all the functions presented above commute.
All these functions can be computed by letting the group $\Aut_{\mathfrak{s}_k}$ presented in Proposition~\ref{propautodd} act on the element
$$\left((f_0,\sqrt{f_v^\mathsf{s}}),\frac{f_v}{\sqrt{f_v^\mathsf{s}}}\right)\in\C^2\times\mathcal{S}\;.$$

In detail, using the notation of Proposition \ref{propautodd}, the action of $\xi$ corresponds to the (left) multiplication by a complex number and the action of $A_\eta$ corresponds to the right $\star$-multiplication.
\end{ex}

\begin{rem}
The previous example shows in practice how to compute all the $k$-th $\star$-roots of a slice regular function $g$, from a given one $f$, whenever $\frac{f_v}{\sqrt{f_v^s}}$ is well defined. In fact, if $\alpha +i\beta$ is a complex 
$k$-th root of $1$, then all the other $k$-th
$\star$-roots of $g$ can be computed as
$f*\left(\alpha+\beta\frac{f_v}{\sqrt{f_v^s}}\right)$ 
if the domain intersects the real axis, or as
$(\alpha+\mathcal{J}\beta)f*\left(\alpha+\beta\frac{f_v}{\sqrt{f_v^s}}\right)$ otherwise.
\end{rem}

We now turn our attention to the case when $k$ is even which, as we saw in Example~\ref{exsubdolo}, is more subtle.

\begin{rem}
If $k$ is even, when computing the kernel of $(\xi,\eta)\mapsto \xi A_\eta$, we obtain
$$\{(1,1),(-1,-1)\}$$
so the map is not injective and its image is just an index $2$ subgroup of $\Aut_{\mathfrak{s}_k}$.
\end{rem}
We notice that if $\lambda,\mu\in U_{2k}$ are primitive $2k$-th roots of $1$, then $\lambda^k=\mu^k=-1$, so 
$$(\lambda A_{\mu})^k=\lambda^k A_{\mu^k}=-1(-\mathrm{Id})=\mathrm{Id}\;.$$
Thanks to this observation we are able to compute the monodromy of $\mathfrak{s}_k$ when $k$ is even.

\begin{propos}

Let $k$ be even and $\lambda,\mu\in U_{2k}$ be primitive roots. If we set
$\Zhe=\lambda A_\mu$, then
$$\Aut_{\mathfrak{s}_k}=\{\xi A_\eta\Zhe^\delta\ :\ (\xi,\eta,\delta)\in U_k\times U_k\times\{0,1\}\}\;,$$
which is isomorphic $\Z_k\times \Z_k$.
\end{propos}
\begin{proof}
The image of the map $(\xi,\eta,\delta)\mapsto \xi A_\eta \Zhe^\delta$ is clearly contained in $\Aut_{\mathfrak{s}_k}$.

This map from $U_k\times U_k\times \Z_2$ to $\Aut_{\mathfrak{s}_{2k}}$ is not a group homomorphism, however, we can factor it through the map from $U_{2k}\times U_{2k}\to \Aut_{\mathfrak{s}_{2k}}$, by sending (injectively) $(\xi,\eta,\delta)$ to $(\xi\lambda^\delta,\eta\mu^\delta)$, so the map $(\xi,\eta,\delta)\mapsto \xi A_\eta \Zhe^\delta$ is a $2$-to-$1$ map.

Again, by cardinality, we conclude that
$$(\xi,\eta,\delta)\mapsto \xi A_\eta \Zhe^\delta$$
is surjective from $U_k\times U_k\times \Z_2$ onto $\Aut_{\mathfrak{s}_k}$.

If now $\xi \in U_k,\ \mu\in U_{2k}$ are primitive roots of unity (of orders $k$ and $2k$ respectively) such that $\mu^2=\xi$, then $\Aut_{\mathfrak{s}_k}$ is generated by
$$\xi, A_{\xi},\ \mu A_{\mu}\;.$$
Moreover, $\xi^k=A_{\xi}^k=1$, $(\mu A_{\mu})^2=\xi A_{\xi}$, so the group $\Aut_{\mathfrak{s}_k}$ can be presented as
$$\langle a,b,c\ \vert\ a^k,\ b^k,\ c^2(ab)^{-1},\ [a,b],\ [a,c],\ [b,c]\rangle$$
which is isomorphic $\Z_k\times \Z_k$ as in the case $k$ odd.
\end{proof}

\begin{ex}
Going back to Example \ref{exsubdolo}, we notice that $U_2=\Z_2$ and, in fact, $g_1=-g_2$; moreover, the primitive roots of $1$ of order $4$ are $\pm \iota$, so that
$$g_3(q)=\mathcal{J}(q)g_1(q)\star \mathcal{J}(g_1(q))$$
and, accordingly, $g_4=-g_3$.
\end{ex}

From the theory of covering maps,
as the group of automorphism of $\mathfrak{s}_k$ acts transitively on the fibers of $\mathfrak{s}_k$, then we can state the following result.

\begin{corol}
The covering map $\mathfrak{s}_k$ is normal (or regular, or Galois).
\end{corol}

Knowing the automorphisms of $\mathfrak{s}_k$, we now pass to analyze $\sigma_k$.

Given $w\in \Omega$, the fiber $\rho^{-1}(w)$ consists of two points $v$ and $\Gamma(v)$ in $W$, where $\Gamma$ is the function defined in Remark~\ref{gamma}; $\rho$ is then a bijection from $\mathfrak{s}_k^{-1}(v)$ to $\sigma_k^{-1}(w)$ and from $\mathfrak{s}_k^{-1}(\Gamma(v))$ to the same set. As $\mathfrak{s}_k$ is normal, the elements of $\Aut_{\mathfrak{s}_k}$ are uniquely defined by their action on $\mathfrak{s}_k^{-1}(v)$, so that each of them corresponds to a different element of the group of permutation of $\mathfrak{s}^-1(v)$ denoted by $\mathrm{Perm}(\mathfrak{s}_k^{-1}(v))$, which is isomorphic, via $\rho$, to $\mathrm{Perm}(\sigma_k^{-1}(w))$.

Given a neighborhood $U$ of $w$ in $\Omega$, which is uniformly covered by $\sigma_k$, we can extend each deck transformation of $\sigma_k^{-1}(w)$ to an automorphism of $\sigma_k^{-1}(U)$ (which is diffeomorphic to the disjoint union of $k^2$ copies of $U$). We can suppose (up to shrinking) that $U$ is also uniformly covered by $\rho$ and that, in turn, each connected component of $\rho^{-1}(U)$ is uniformly covered by $\mathfrak{s}_k$.

The restriction of the automorphisms in $\Aut_{\mathfrak{s}_k}$ to the connected components of $\mathfrak{s}_k^{-1}(\rho^{-1}(U))$ correspond, via $\rho$, to the extension of the deck transformations of $\sigma_k^{-1}(w)$ to $\sigma_k^{-1}(U)$.

\begin{propos}The group $\Aut_{\sigma_k}$ is isomorphic to $\Z_k$. In particular, $\sigma_k$ is never regular.\end{propos}

\begin{proof}
By the discussion above, given $F\in\Aut_{\sigma_k}$, there is $G\in\Aut_{\mathfrak{s}_k}$ such that 
$$F\circ\rho=\rho\circ G\;.$$
Therefore, $\rho\circ G=\rho\circ G\circ \Gamma$ and, as $G$ is an automorphism, this happens if and only if $G\circ \Gamma=\Gamma\circ G$. 
Given $(u,s)\in\C^2\times\mathcal{S}$, there exist $\xi,\eta\in U_k$ such that $G(u,s)=(\xi A_\eta u, s)$, so we want that
$$(\xi QA_\eta u, -s)=(\xi A_\eta Qu, -s)$$
for all $(u,s)$, where 
$$Q=\begin{pmatrix}1&0\\0&-1\end{pmatrix}\,.$$

In conclusion, we need that $A_\eta Q=QA_\eta$. Writing $\eta=a+\imath b$, we have
$$\begin{pmatrix}a&-b\\-b&-a\end{pmatrix}=\begin{pmatrix}a&b\\b&-a\end{pmatrix}$$
which holds if and only if $b=-b$, i.e. $b=0$.

For $k$ odd, this implies $\eta=1$, for $k$ even, we obtain $\eta=\pm1$. In both cases, we obtain that $G$ is of the form $$G(u,s)=G(\xi u,s)$$
so the corresponding $F\in\Aut_{\sigma_k}$ is of the form $F(z)=\xi z$, with $z\in\Omega_k\subset\C^4$ and $\xi\in U_k$, which obviously form, under composition, the group $\Z_k$.

Given that $\sigma_k$ is a covering map of degree $k^2$, this implies that it is not a regular covering.
\end{proof}

\begin{ex}
Comparing with Example~\ref{exdeg3}, we notice that for any $\ell=0,1,2$, the orbit $g_\ell$, its orbit in $\Aut_{\sigma_k}$ is given by $\{g_\ell,g_{\ell,1},g_{\ell,2}\}$.
\end{ex}

We have already exploited the existence of the automorphisms $z\mapsto \xi z$ in the discussion after Theorem \ref{teo_kroots}, to obtain, in the odd case, all the roots from the intrinsic ones. However, if our slice function corresponds to a stem function defined on a symmetric simply connected open domain (or on a union of two symmetric simply connected open domain) $\mathcal{U}\subseteq\C$, we completely describe all the $k^2$ roots.  
We recall here that, as already stated, the hypothesis of simply connectedness is not removable due to standard covering theory.
In the last part of this paper we provide algebraic techniques to compute all the $k$-th $\star$-roots of a given slice functions, starting from one of them.

\subsection{Permutations of $k$-th roots}
Suppose now that we are given a slice function $f:U\to\H$ and its stem function $F:\mathcal{U}\to\C^4\simeq\C\otimes\H$, with $\mathcal{U}$ simply connected and $F(\mathcal{U})\subseteq\Omega$.

By Proposition \ref{prplifts}, the set
$$\mathcal{G}=\{G:\mathcal{U}\to\Omega_k\ :\ \sigma_k\circ G=F\}$$
contains $k^2$ elements. By the properties of covering automorphisms we have that $\Aut_{\sigma_k}$ acts on $\mathcal{G}$ by  (post-)composition, partitioning it in $k$ orbits with $k$ elements each. Explicitly, given $\xi\in U_k$, the map $G\mapsto\xi G$ is a permutation of $\mathcal{G}$ without fixed points (unless $\xi=1$).

\begin{propos}\label{prp_noaction}Let $G\in \mathcal{G}$ and fix $H:\mathcal{U}\to W_k$ such that $\rho\circ H=G$. Then for each $\tau \in\Aut_{\mathfrak{s}_k}$ the function
$$G_\tau=\rho\circ\tau\circ H$$
belongs to $\mathcal{G}$.\end{propos}
\begin{proof}
We have that $\sigma_k\circ  G_\tau=\sigma_k\circ\rho\circ\tau\circ H$. We know that $\sigma_k\circ\rho=\rho\circ \mathfrak{s}_k$ and $\mathfrak{s}_k\circ\tau=\mathfrak{s}_k$, so
$$\sigma_k\circ G_\tau=\rho\circ \mathfrak{s}_k\circ \tau\circ H=\rho\circ \mathfrak{s}_k\circ H=\sigma_k\circ\rho\circ H=\sigma_k\circ G\;.$$
Therefore, $G_\tau\in\mathcal{G}$.
\end{proof}

Thanks to Proposition~\ref{prp_noaction}, we recover all the $k^2$ functions in $\mathcal{G}$ from one of them.
We note that the map $\tau\mapsto G_\tau$ is not an action of $\Aut_{\mathfrak{s}_k}$ on $\mathcal{G}$, as the different possible choices of $H$ such that $\rho\circ H=G$ result in different definitions for $G_\tau$ and there is not a coherent way to choose such an $H$ for all the $G\in\mathcal{G}$.
In order to obtain a result in this direction, we now reinterpret some computation given in the previous section with
the language of cover automorphisms (see Remark~\ref{remarkslice}).

\begin{defin}Given $G:\mathcal{U}\to\C^4$, we define $TG:\mathcal{U}\to\C^4$ as $TG(z)=\overline{G(\bar{z})}$.\end{defin}

The previous definition reinterpret in the language of curves in $\C^4$
what previously was defined in the
context of $\C\otimes\H$ (see Remark~\ref{remarkslice}).

\begin{lemma}\label{lmm_TG}We have that
\begin{enumerate}
    \item $TTG=G$;
    \item if $\xi\in U_k$, $T(\xi G)=\bar{\xi}TG$;
    \item if $G\in\mathcal{G}$, then $TG\in\mathcal{G}$;
    \item the subgroup of $\mathrm{Perm}\mathcal{G}$ generated by $\Aut_{\sigma_k}$ and $T$ is isomorphic to the dihedral group.
\end{enumerate}
\end{lemma}
\begin{proof}
The first and second statements are obvious. 

Given $G\in\mathcal{G}$, $\sigma_k(G(z))=F(z)$ so $T(\sigma_k\circ G)=TF=F$ (because $F$ is the stem function of a slice function). Moreover, $\sigma_k$ has polynomial components with real coefficients, so
$$\sigma_k TG(z)=\sigma_k(\overline{G(\bar{z})})=\overline{\sigma_k(G(\bar{z}))}=T\sigma_k\circ G(z)=F(z)\;,$$
which implies $TG\in\mathcal{G}$.

$\Aut_{\sigma_k}\cong U_k\cong \Z_k$ acts on $\mathcal{G}$ by scalar multiplication. Given $\xi$ a primitive $k$-th rooth of unity, we have that the group generated by $\xi$ and $T$ has the following presentation
$$\langle \xi,\ T\ \vert\ \xi^k, T^2, \xi T\xi=T\rangle$$
(once we notice that $\bar{\xi}=1/\xi$),
which is the standard presentation of the dihedral group.
\end{proof}

The last statement of Lemma \ref{lmm_TG}  defines an action of the dihedral group on $\mathcal{G}$.

In the next theorem we show, in the case $k$ odd, how to recover all the stem functions in $\mathcal{G}$ starting from one of them.

\begin{teorema}\label{teo_genrootsodd}If $k$ is odd and $F$ is a fixed point for $T$ (i.e. if it is a stem function), then there exist $k$ functions $G:\mathcal{U}\to\C^4$ such that
$$\sigma_k\circ G=F\qquad TG=G\;.$$
\end{teorema}

\begin{ex}
In  Example~\ref{exdeg3} the functions fixed by $T$ are the stem functions of $g_0,g_1$ and $g_2$.
\end{ex}

\begin{proof}
If $k$ is odd, there is always at least one element of $\mathcal{G}$ such that $TG=G$, because $T$ is an involution. If we consider the elements $\tau$ of $\Aut_{\mathfrak{s}_k}$ induced by the matrices $A_\eta$, with $\eta\in U_k$, then the corresponding functions $G_\tau$ obtained via Proposition \ref{prp_noaction} are all fixed points of $T$: $G_\tau=\rho\circ\tau\circ H$ and $H(z)=(u(z), s(z))$, so $\tau\circ H(z)=(A_\eta u(z), s(z))$, which means that
$$TG_{\tau}=\overline{\rho(A_\eta u(\bar{z}), s(\bar{z}))}= \rho(A_\eta \overline{u(\bar z)},\overline{s(\bar{z})})$$
because $A_\eta$ is a real matrix, and, as $TG=G$, $\overline{u(\bar{z})}=u(z)$ and $\overline{s(\bar{z})}=s(z)$.

Therefore, $k$ elements of $\mathcal{G}$ are fixed points for $T$.
\end{proof}

\begin{rem}
Notice that the previous theorem is not a rephrased version of Theorem~\ref{teo_kroots}. In fact,
in this last result, we are not assuming that the domain intersects the real line and yet we obtain that, at the level of stem functions, $k$ solutions are \textit{Schwarz-symmetric}.
\end{rem}

As in the previous section, the result for $k$ even needs some additional efforts: let $F':\mathcal{U}\to \C^2\times\mathcal{S}$ be such that $\rho\circ F'=F$; define
$$\mathcal{G}'=\{G':\mathcal{U}\to\C^2\times\mathcal{S}\ :\ \mathfrak{s}_k\circ G'=F'\}$$
and define $TG'(z)=\overline{G'(\bar{z})}$, where, if $z=((u_0,u_1),s), then$ $\overline{z}=\overline{((u_0,u_1),s)}=((\bar{u}_0, \bar{u}_1),\bar{s})$.

If $F$ is a stem function, then $TF'=F'$.

\begin{lemma}\label{lmm_TGp}We have that
\begin{enumerate}
    \item $TTG'=G'$
    \item if $G'\in\mathcal{G}'$, then $TG'\in\mathcal{G}'$
    \item for each $G'\in\mathcal{G'}$, there exists $\tau_{G'}\in\Aut_{\mathfrak{s}_k}$ such that $TG'=\tau_{G'}\circ G'$
    \item $\tau_{G'}$ is always of the form $(u,s)\mapsto (\xi u, s)$.
\end{enumerate}
\end{lemma}
\begin{proof}
The first and the second statements are analogous to the ones in Lemma \ref{lmm_TG}.

The third statement follows from the uniqueness of the lifts through a covering map.

For the fourth, we note that, if $\tau(u,s)=(A_\eta u, s)$, then $T(\tau\circ G')=\tau\circ TG'$. By \textcolor{blue}{$(1)$}, $TTG'=G'$, so, if $TG'=\tau\circ G'$, we have that 
$$TTG'=T(\tau\circ G')=\tau\circ TG'=\tau\circ\tau\circ G'$$
i.e. $\tau\circ\tau$ is the identity, which implies $A_\eta=\pm I$ (so, \textcolor{blue}{$\tau$} is of the form $(u,s)\mapsto (\xi u,s)$).
\end{proof}

\begin{rem}
Given $\eta\in U_k$, let $\tau_\eta\in\Aut_{\mathfrak{s}_k}$ be of the form $\tau_\eta(u,s)=(\eta u,s)$, then
$$T(\tau_\eta\circ G')=\tau_{\bar{\eta}}TG'=\tau_{\bar{\eta}}\circ \tau_{G'}\circ G'=\tau_{\bar{\eta}^2}\circ \tau_{G'}\circ \tau_\eta\circ G'\;.$$
Therefore, if $G''=\tau_\eta\circ G'$, then $$\tau_{G''}=\tau_{\bar{\eta}^2}\circ\tau_{G'}\;.$$
On the other hand, if $\tau(u,z)=(A_\eta u, s)$ and $G''=\tau\circ G'$, then $\tau_{G''}=\tau_{G'}$, as shown in the end of the proof of Lemma \ref{lmm_TGp}.
\end{rem}

Therefore, given $G'\in\mathcal{G}'$, $\tau_{G'}(u,s)=(\xi u,s)$ with $\xi\in U_k$, we can always find $\tau$ in $\Aut_{\mathfrak{s}_k}$ such that $\tau\circ\tau=\tau_{\bar{\xi}}$, so that, setting $G''=\tau\circ G'$, we have
$$\tau_{G''}=1\;.$$

So, we have found an element of $\mathcal{G}'$ such that $TG''=G''$; we also know that, for all $\eta\in U_k$, given
$\tau\in \Aut_{\mathfrak{s}_k}$, $\tau\circ G''$ is again a fixed point for $T$.

By composing with $\rho$ and considering also the previous Theorem, we have proved the following
\begin{teorema}\label{teo_genrootseven}If $F$ is a fixed point for $T$ (i.e. if it is a stem function), then there exist $k$ functions $G:\mathcal{U}\to\C^4$ such that
$$\sigma_k\circ G=F\qquad TG=G\;.$$
\end{teorema}

\begin{ex}
In  Example~\ref{exsubdolo} the functions fixed by $T$ are the stem functions of $g_3$ and $g_4$.
\end{ex}

\bibliographystyle{plain}
\bibliography{powerslice}{}
    
\end{document}